\documentclass[10pt]{article}
\usepackage{amsmath} 
\usepackage{amssymb} 
\usepackage{amsfonts}
\usepackage{amsthm}
\usepackage{amscd}
\usepackage{pstricks}
\usepackage{graphicx}
\usepackage{mystyle} 
\usepackage[latin1]{inputenc}

\theoremstyle{plain}      
\newtheorem{theorem}{Theorem}[section]
\newtheorem{prop}[theorem]{Proposition}
\newtheorem{lemma}[theorem]{Lemma}
\newtheorem{cor}[theorem]{Corollary}

\theoremstyle{definition} 
\newtheorem{definition}[theorem]{Definition}
\newtheorem{remark}[theorem]{Remark}
\newtheorem{example}[theorem]{Example}

\numberwithin{equation}{section}

\newcommand{\M}{\mathcal{M}}
\newcommand{\C}{\text{Crit}}

\begin{document}

\begin{center}
\vspace*{1cm}
\Huge{$S^1$-Equivariant Morse Cohomology}\\[2cm]
\large{Diplomarbeit}\\[1cm]
\large{Humboldt-Universit\"at zu Berlin}\\[0.2cm]
\large{Mathematisch-Naturwissenschaftliche Fakult\"at II}\\[0.2cm]
\large{Institut f\"ur Mathematik}


\end{center}
\vspace{4.5cm}

\begin{tabular}{ll}
eingereicht von: &Marko Josef Berghoff\\[0.2cm]
geboren am: &15.~Februar 1982 in Frankfurt/Main\\[0.2cm]
1.\,Gutachter: &Prof.~Dr.~Klaus Mohnke\\[0.2cm]
2.\,Gutachter: &Prof.~Dr.~Helga Baum \\[0.2cm]
\end{tabular}\\
\\[0.4cm]
\begin{center} Berlin, den 22.~Dezember 2009
\end{center}
\setcounter{tocdepth}{2}

\tableofcontents

\newpage

\section{Introduction}
\label{intro}

Equivariant cohomology is a relatively new cohomology theory playing an important role in algebraic geometry and (transformation-)group theory. It has also many applications in modern symplectic geometry (Audin \cite{Au}, Libine \cite{Li}, which is also a good introduction to equivariant cohomology) and theoretical physics (Szabo \cite{Sz}). Roughly speaking, for a $G$-space $X$ equivariant cohomology is a cohomology theory lying somewhere between the ordinary cohomology of $X$ and the group cohomology of $G$. It was first introduced in 1960 by A. Borel \cite{Bo} for the aim of studying transformation groups.

\begin{definition}
Let $H^*$ denote a cohomology functor (say singular, but any other satisfying the Eilenberg-Steenrod axioms would work as well - with coefficients in some ring $\mathcal{R}$, which we omit in the following). Let $X$ be a topological space and let $G$ be a topological group acting continuously on $X$. Associated to $G$ there is the universal bundle of $G$, the principal $G$-bundle

\begin{equation*}
\begin{CD}
 EG \\
 @VV p V \\
 BG.
\end{CD}
\end{equation*}
Here $EG$ is the unique (up to weak homotopy) contractible topological space endowed with a free $G$-action and $BG:=EG/G$ is called the \textit{classifying space of $G$} (see Steenrod \cite{St} for universal bundles, \cite{Mi1} for Milnor's construction of $EG$).

The \textit{equivariant cohomology} $H^*_G(X)$ of $X$ is defined as the cohomology of the total space of the associated fiber bundle

\begin{equation} \label{bundle} 
\begin{CD} 
 \hspace{-0.8 cm} M \to X_G  \\
 @VV \pi V \\
  BG,
\end{CD}
\end{equation}
where $X_G:=(EG\times X)/G$ is the \textit{homotopy quotient} of $X$ with respect to the diagonal action of $G$ on $EG\times X$:

\begin{equation*} 
H^*_G(X):=H^*(X_G).
\end{equation*}

\end{definition}

The idea is to have a cohomology theory reflecting the topological properties of $X$, $G$ and the space of $G$-orbits $X/G$. If the action is free, $H^*(X/G)$ is a good candidate (as it turns out in this case $H^*_G(X)\cong H^*(X/G)$), but for non-free actions $X/G$ can be very pathological as already simple examples show:

\begin{example}
Let $S^1$ act on $\mathbb{C}^2$ by

\begin{equation*}
s.(z_1,z_2):=(s^mz_1,s^nz_2) \quad \text{(m,n $\in \mathbb{N}$ relatively prime)}.
\end{equation*}

This action is not free (for example $(z_1,0)$ has a non-trivial stabilizer) and the quotient space is a manifold with singularities, the \textit{weighted projective space}. The study of such geometrical objects led to the notion of \textit{orbifolds} (cf. Satake \cite{Sat}).
\end{example}
So instead of considering $X/G$ one works with $X_G$. Since the action of $G$ on $EG\times X$ is free, $X_G$ again is a ``nice'' topological space.

But for $G=S^1$ (and most other interesting groups) $X_G$ is infinite-dimensional: $ES^1=S^{\infty}$ is the infinite-dimensional sphere, the direct limit $\varinjlim S^n$ of the directed system $(S^n,\iota_{ij})$, where the maps $\iota_{ij}$ are the inclusions $S^i \to S^j$. As a consequence, if $M$ is a smooth finite-dimensional $G$-manifold with $G$ a Lie group, there are two possible ways using Morse or de Rham theory to get $H^*(M_G)=H_G^*(M)$: Either by using finite-dimensional approximations $M^k_G$ to the homotopy quotient $M_G$ and the fact that $H^n(M^k_G)=H^n(M_G)$ for $k$ large or by adjusting these classical ideas to the equivariant setting.
\\

A de Rham theoretic construction is due to H. Cartan \cite{Ca,Ca2}: Here $A \otimes \Omega(M)$ with $A$ an acyclic algebra of \textit{type C} (this is the appropriate analogon to freeness of the $G$-action on $EG$) is used as an algebraic substitute for the complex of differential forms on $EG\times M$. $\Omega(M_G)$ is then modelled by the subcomplex of \textit{basic forms} on $A \otimes \Omega(M)$ (a form $\omega$ is called \textit{basic} if it is invariant under the $G$-action, i.e. $\sigma_g^*\omega=\omega\ \forall g\in G$, and kills each vector $v$ tangent to a $G$-orbit, i.e. $\iota_v\omega \equiv0$). This complex of equivariant forms is then endowed with a twisted de Rham differential $d_{\text{eq}}$ incorporating the induced $G$-action on forms and an equivariant version of the de Rham Theorem states that $H_*((A\otimes \Omega(M))_{\text{basic}},d_{\text{eq}})\cong H^*_G(M)$. For details and other formulations of the equivariant de Rham complex, see the original work by Cartan, or Libine \cite{Li} as well as the book by Guillemin and Sternberg \cite{GS}. 

On the other hand, there are also Morse-theoretic approaches to equivariant cohomology. Austin and Braam \cite{AB} use Morse-Bott theory for a $G$-equivariant (i.e. $G$-invariant) function $f$ on $M$ together with equivariant differential forms on the critical submanifolds of $f$ to construct a complex whose homology is isomorphic to $H^*_G(M)$. The chains of this complex are equivariant forms on $\C(f)$ and the differential is given by integrating these forms over gradient flow lines of $f$. 

In some way similiar to our approach is a special case of Hutchings ``Floer homology of families'' \cite{Hu}, where ``family'' means a set of equivalent objects parametrized by a smooth manifold, e.g.~(finite-dimensional approximations to) the bundle $\pi$ in $\eqref{bundle}$. $H^*(M^n_G)$ is the homology of a Morse complex constructed by studying critical points and flow lines of a vector field $V+W$ on $M^n_G$. Here $W$ is the horizontal lift of the gradient vector field of a Morse function $f:BG^n \to \mathbb{R}$ and $V$ is the fiberwise gradient vector field of Morse functions $f_x:M_x\to \mathbb{R}$ on the fibers over $x\in \C(f)$.
\\

From now on let $G=S^1$. In this thesis we surpass the problem of $M_{S^1}$ being infinite-dimensional by reducing the computation to the finite-dimensional fibers $M$ of the bundle $\pi$. On these fibers we use a deformed Morse complex. 

This is due to Yuri Chekanov who introduced this approach to equivariant cohomology in talks given at the MPI Leipzig and the ETH Z\"urich in 2004 and 2005, respectively. 

His idea is to deform the Morse complex associated to a Morse function $f:M\to \mathbb{R}$ by incorporating the $S^1$-action on $M$ into the definition of the coboundary operator: $d$ counts not only usual gradient flow lines of $f$, but also such ones that are allowed to ``jump'' along orbits of the action for finite time intervals. The benefit of this is that we are working the whole time on a finite dimensional space. This allows us to use the geometric and more intuitive methods from the theory of dynamical systems for studying these ``jumping'' gradient flow lines of $f$ (cf. Weber \cite{We}) instead of the heavyweight functional analytic apparatus used in Floer theory, which basically is Morse theory on infinite-dimensional manifolds (see Salamon \cite{Sal} for Floer theory, Schwarz \cite{Sch} for a Floer-type approach to Morse homology). 

We get a complex freely generated by the critical points of $f$ tensored with the polynomial ring $\mathcal{R}[T]=H^*(BS^1)=H^*(\mathbb{CP}^{\infty})$, the $S^1$-equivariant cohomology of a one-point space. The differential operator is defined by counting ``jumping'' gradient flow lines which are modelled as follows:\\
A \textit{$k$-jump flow line} is a solution of the ODE

\begin{equation*}
\dot{u}(t)=V_t(u(t)).
\end{equation*}
Here $V_t$ is a time-dependent vector field associated to the gradient of a homotopy $f_t$ satisfying
\begin{equation*}
f_t(x)=\begin{cases}
    f(x) & \text{if $ t < t_1 $ },\\
    f(s_1.x) & \text{if $t_2 \leq t < t_3$ }, \\
    f(s_2.s_1.x) & \text{if $ t_4 \leq t < t_5 $}, \\
    \qquad \cdots & \\
    f(s_k. \cdots .s_1.x) & \text{if  $t_{2k} \leq t$, }
  \end{cases}
\end{equation*}
for some $t_1 < \ldots < t_{2k} \in \mathbb{R}$ and $s_i \in S^1$.

This will be explained in detail in Chapter 3. Basically, our approach translates into a special case of an idea of Frauenfelder \cite{Ff} called ``flow lines with cascades''. Here the cohomology of a manifold $M$ is derived from studying gradient flow lines of a Morse-Bott function $f$ on $M$ (the ``cascades'') and gradient flow lines of a Morse function $h$ on $\C(f)$. Applying this idea to the bundle $\pi$ in $\eqref{bundle}$ and using its properties together with the special structure of its fibers and base space $\mathbb{CP}^{\infty}$ the cascades translate into jumps along orbits of the $S^1$-action on $M$.

Hope is that, as in the relation of Morse and Floer theory, our construction serves as a toy model for a similiar approach to $S^1$-equivariant Floer cohomology.

This thesis is organized as follows: In the next chapter we review basic Morse and Morse-Bott theory, the gradient flow line approach to Morse homology and introduce the Morse complex associated to flow lines with cascades which will be used in the end to justify our construction. In Chapter 3 the equivariant Morse cochain groups $CM_{S^1}^*$ and the equivariant Morse differential $d_{S^1}$ are defined; this involves the definition of the moduli spaces of $k$-jump flow lines. In the fourth chapter we continue examining the properties of these moduli spaces: We show by using the theory of dynamical systems, that they carry the structure of finite-dimensional manifolds which admit a natural compactification. Furthermore, there is a complementary concept to compactification: The gluing map, which glues flow lines from different moduli spaces to a flow line in some higher dimensional moduli space. Putting these facts together we conclude in the Chapter $5$ that $(CM_{S^1}^*,d_{S^1})$ is actually a cochain complex and, using the idea of flow lines with cascades, we show that $H_*(CM_{S^1}^*,d_{S^1})\cong H^*_{S^1}(M)$. After that we finish with an outlook.
\\
 
I want to thank my family for their support and my advisor Klaus Mohnke for suggesting me this interesting topic and always having time for my numerous questions and many fruitful discussions. 

Furthermore, I have to give special credit to Yuri Chekanov, because the idea of approaching $S^1$-equivariant cohomology with jumping flow lines is due to him, and I benefited greatly from his visit to Berlin, where he took the time to explain his ideas to me.

\section{Preliminaries}
\label{prelims}

In this chapter we review basic Morse and Morse-Bott theory. A good reference for this is the book ``Lectures on Morse Homology'' by Banyaga and Hurtubise \cite{BH2}. We start with some elementary facts about Morse theory, then introduce the gradient flow line approach to Morse homology. Then we continue with some Morse-Bott theory and explain how flow lines with cascades are used to compute $H^*(M)$.

From now on throughout this thesis let $(M,g)$ denote a smooth $n$-dimensional closed manifold with Riemannian metric $g$ endowed with a smooth $S^1$-action. Assume without loss of generality that $M$ is connected.
\subsection{Morse theory}
\label{mt}

The basic idea of Morse(-Bott) theory, originally due to M. Morse \cite{Mo}, is to extract information about the topology of $M$ by studying the local and global behaviour of smooth functions $f:M\to \mathbb{R}$. 

\begin{definition}
A function $f \in C^{\infty}(M,\mathbb{R})$ is called \textit{Morse}, if all its critical points are non-degenerate. A critical point is called non-degenerate, if the matrix associated to the symmetric bilinear form $D^2f(x): T_xM \times T_xM \to \mathbb{R}$, the Hessian of $f$ at $x$, is non-singular. In local coordinates $(p^i)$ around $x$ the entries of the Hessian matrix $H$ are given by \smash{$h_{ij}=\frac{\partial^2f}{\partial p^i \partial p^j}(x)$}.
\end{definition}

Note: The non-degeneracy condition implies that $\C(f)$ is an isolated set and from compactness of $M$ it follows that it is finite.

\begin{definition}[The Morse index]
For a Morse function $f$ the \textit{Morse index} $\mu(x)$ of a critical point $x$ is defined as the number of negative eigenvalues of the Hessian of $f$ at $x$. 
\end{definition}

We state four classical theorems of Morse theory; for proofs we refer to Milnor \cite{Mi2}:

\begin{lemma}[Morse Lemma]\label{ml}
For $x\in \C(f)$ there are local coordinates $(p^i)$ around $x$ such that 
\begin{align*}
f(p)=&f(x) - \sum_{i=1}^{\mu(x)}(p^i)^2 + \sum_{j > \mu(x)}^n (p^j)^2, \\
&p^i(x)=0 \quad \forall i=1, \dots, n.
\end{align*}
\end{lemma}

\begin{prop} 
Set $M_c:= f^{-1}((-\infty , c])$. Then, if there is no critical value in $[a,b]$, $M_a$ is diffeomorphic to $M_b$. It is even a deformation retract of $M_b$.
\end{prop}

\begin{prop}
If there is one critical point of index $k$ in $[a,b]$, then $M_b$ is homotopy equivalent to $M_a$ with a $k$-cell attached.
\end{prop}

This already implies the ``Morse inequalities'': 

\begin{theorem}[The Morse inequalities]
Let $b_k$ be the Betti numbers of $M$ and $c_k$ be the number of critical points of $f$ with index $k$, then for all $k=1, \ldots,n$: 
\begin{align*}
c_k - c_{k-1} + c_{k-2} - \dots +(-1)^k c_0 &\geq b_k -b_{k-1} + b_{k-2} - \dots + (-1)^k b_0, \\
\sum_{k=0}^n (-1)^k c_k &= \sum_{k=0}^n (-1)^k b_k = \chi(M).
\end{align*}
\end{theorem}

These theorems already indicate the strong relationship between the topology of $M$ and the structure of $\C(f)$. Further investigations into the subject, mostly by Milnor, Thom, Smale and later Witten led to the notion of ``Morse homology''; for a brief history of this development see Bott \cite{Bott}. \\
We now present the gradient flow line approach to Morse homology using the theory of dynamical systems from Weber \cite{We} based on the ideas of Witten \cite{Wi}, excluding the concept of orientation, i.e. using only $\mathbb{Z}_2$-coefficients. The reason for this will be discussed later in Chapter \ref{emcoh}. 
Proofs of the following theorems can be found in Weber \cite{We} or in Chapter $4$, where we proof slight generalizations of these theorems by essentially the same methods.

There is no big difference in the construction of Morse homology and cohomology: Since Morse chains are elements of the free abelian group generated by the critical points of $f$ the map $ CM_*(f,g) \to CM^*(f,g), \ x \mapsto h_x$ with 
\begin{equation*}
h_x(y)=\begin{cases}
    1 & \text{if $y=x$ },\\
    0 & \text{ otherwise, }
  \end{cases}
\end{equation*}
is a bijection between the Morse chains and cochains and the differential differs only by a sign change in the equation for gradient flow lines, i.e. $V \to -V$ in \eqref{eq1} below.\\
Let $V:=\nabla_g f$ be the gradient vector field of $f$ and consider the dynamical system:
\begin{align}
\dot{u}(t)=&V(u(t)), \label{eq1} \\
u(0)=&u_0 \in M. \nonumber
\end{align}

$M$ is closed, therefore solutions of \eqref{eq1} are defined for all $t \in \mathbb{R}$ and since the zeros of $V$ are precisely the critical points of $f$ the limits $u^-:= \lim_{t \to -\infty}u(t)$ and $u^+:= \lim_{t \to \infty}u(t)$ are elements of $\C(f)$. Furthermore, $u^- \neq u^+ $ because $V=\nabla_g f$ and therefore $f$ is strictly increasing along $u$.

\begin{theorem}[Stable Manifold Theorem]
Let $\Phi: M \times \mathbb{R} \to M,\ (p,t) \mapsto \Phi_t(p)$ denote the flow of $V$. Then for $x \in \C(f)$ the \textit{(un-)stable manifolds of $x$}, defined by 
\begin{align*}
W^u(x)&:=\{ q \in M | \lim_{t \to -\infty} \Phi_t(q)=x \}, \\
W^s(x)&:=\{ q \in M | \lim_{t \to \infty} \Phi_t(q)=x \}.
\end{align*}
are open submanifolds of $M$, diffeomorphic to $\mathbb{R}^{n-\mu(x)}$ and $\mathbb{R}^{\mu(x)}$, respectively. 
\end{theorem}

\begin{remark} \label{tangent}
For $x\in \C(f)$ there is the following description of the tangent space $T_xW^i(x) \ (i=u,s)$: Since $x$ is a non-degenerate critical point, $T_xM$ splits into $E^u \oplus E^s$ consisting of the eigenvectors associated to negative and positive eigenvalues of $H(f)(x)$. Moreover, the flow of $D\nabla_g f(x)$ viewed as a linear vector field on $T_xM$ is given by the linear map $A_t:=D\Phi_t(x)= \exp{(t D\nabla f(x))}$ and if $\lambda$ is an eigenvalue of $H(f)(x)$, then $e^{\lambda t}$ is an eigenvalue of $A_t$ and both have the same eigenspaces. Hence $E^i=T_xW^i(x)$. Furthermore, the spaces $E^i$ are invariant under $A_t$ and $A_t$ is a strict contraction on $E^s$ and a strict dilatation on $E^u$.
\end{remark}

The union \smash{$\bigcup_{x \in \C(f)} W^i(x)$} forms an open cover of $M$. So understanding the closure of each $W^i(x)$ (i.e. knowing the attaching maps) would give a cell decomposition of $M$ from which the homology of $M$ could be computed. Unfortunately in general their closure is quite complicated and singular (indicating that there is information about the topology of $M$ encoded). 

On the other hand, the manifolds $W^i(x)$ are also used to describe the spaces of solutions of equation \eqref{eq1} which flow from one critical point to another:

\begin{definition}
For $x,y \in \C(f)$ define 
$$
\M(x,y):=W^u(x) \cap W^s(y),
$$
the \textit{connecting space of $x$ and $y$} consisting of all flow lines emanating at $x$ and ending at $y$.
\end{definition}

Observe that there is a free $\mathbb{R}$-action on $\M(x,y)$ given by shifting $u(t)$ to $u(t+\tau)$ for $\tau \in \mathbb{R}$.

\begin{definition}
The \textit{moduli space of flow lines from $x$ to $y$} is defined as
$$
\widetilde{\M}(x,y):=\M(x,y) / \mathbb{R}.
$$
\end{definition}

\begin{remark}
Equivalently one could define $\widetilde{\M}(x,y)=\M(x,y) \cap f^{-1}(a)$, where $a \in (f(x),f(y))$ is a regular value and the identification of moduli spaces associated to different regular values is provided by the flow.
\end{remark}

\begin{theorem} For a generic pair $(f,g)$ all stable and unstable manifolds intersect transversally, so that $\M(x,y)$ and $\widetilde{\M}(x,y)$ are $(\mu(y)-\mu(x))$- and $(\mu(y)-\mu(x)-1)$-dimensional submanifolds without boundary of $M$.
Such a pair $(f,g)$ is called \textit{Morse-Smale}.
\end{theorem}

This theorem is due to Smale. In \cite{Sm} he showed that either $f$ or the vector field $V$ can be $C^1$ approximated by $\tilde{f}$ or $\tilde{V}$ respectively to obtain transversality of all intersections. Equivalently one could also change the metric $g$ - this is not true for Morse-Bott functions: cf. Latschev \cite{Lat}, where some counter-examples are presented. For a discussion how the change of $g$ and $\nabla_g f$ is related see Wall \cite{Wall}.\\
Hence, for $\mu(x)=\mu(y)-1$ the moduli space $\widetilde{\M}(x,y)$ is $0$-dimensional. Moreover, it is compact and therefore a finite set. For arbitrary $x,y \in \C(f)$ all higher-dimensional moduli spaces are naturally compactifiable. The compactification of $\widetilde{\M}(x,y)$ is given by adding so-called ``broken flow lines'' to $\widetilde{\M}(x,y)$:

\begin{definition} \label{broken}
For $p \in M$ let \smash{$ O(p):=\bigcup_{t\in\mathbb{R}}\Phi_t(p)$} be the flow line through $p$. 
A sequence $(u_k)_{k\in \mathbb{N}} \subset \widetilde{\M}(x,y)$ \textit{converges to a broken flow line $(v_1, \dots , v_l)$ of order $l$} iff there exist critical points $x_0=x, x_1, \dots , x_{l-1}, x_l=y$ such that 
\begin{gather*}
v_i \in \M(x_{i-1},x_i)\ \forall i=1,\ldots,l\\
\text{and} \\
O(u_k) \to v_1 \cup \dots \cup v_l \text{ as } k \to \infty,
\end{gather*}
where convergence with relation to the Riemannian distance $d$ is meant.

A subset $K \subset \widetilde{\M}(x,y)$ is called \textit{compact up to broken flow lines} iff every sequence $(u_k)_k \subset \widetilde{\M}(x,y)$ possesses a subsequence converging in the above sence.
\end{definition}

\begin{theorem} \label{mcomp}
For $x,y \in \C(f)$ and a Morse-Smale pair $(f,g)$, $\widetilde{\M}(x,y)$ is compact up to broken flow lines of order at most $\mu(y)-\mu(x)$.
\end{theorem}

Hence, for $k=\mu(y)-\mu(x)$ we have the following description for the topological boundary of the connected components of the compactified moduli space $\widetilde{\M}(x,y)$:
\begin{equation*}
\partial \widetilde{\M}(x,y)= \bigcup_{\substack{z_1, \dots ,z_{k-1} \in \C(f) \\ z_i\neq z_j } } \widetilde{\M}(x,z_1) \times \widetilde{\M}(z_1,z_2) \times \ldots \times \widetilde{\M}(z_{k-1},y).
\end{equation*}
There is an opposite concept to convergence to broken flow lines called gluing: Given a flow line from $x$ to $y$ and one from $y$ to $z$ the gluing map $\Diamond_{\rho}$ produces a flow line in the higher-dimensional moduli space $\widetilde{\M}(x,z)$:

\begin{theorem}
Let $(f,g)$ be Morse-Smale and $x,y,z \in Crit(f)$, with Morse indices $\mu(x)=k=\mu(y)-1=\mu(z)-2$. There is $\rho_0 > 0$ and a map

\begin{equation*}
\Diamond: \widetilde{\M}(x,y) \times [\rho_0,\infty) \times \widetilde{\M}(y,z) \to \widetilde{\M}(x,z), \ (u,\rho,v) \mapsto u\Diamond_{\rho}v, 
\end{equation*}
such that
\begin{equation*}
u\Diamond_{\rho}v \to (u,v) \text{ for } \rho \to \infty 
\end{equation*}
and no other sequence in $\widetilde{\M}(x,z) \setminus u\Diamond_{[\rho_0,\infty)}v$ converges to $(u,v)$.
\end{theorem}

The last statement in the theorem above is crucial: Together with the Compactness Theorem this shows that for index difference $\mu(y)-\mu(x)=2$ the broken flow lines of order $2$ passing intermediate critical points $z_i$ are precisely the boundary components of the (compactification of the) one-dimensional manifold $\widetilde{\M}(x,y)$.
\begin{definition}[The Morse(-Thom-Smale-Witten) complex]
Let $CM^*(f,g)$ be the free $\mathbb{Z}_2$-module generated by $\C(f)$. Define a differential operator on a generator $x\in \C(f)$ by
\begin{equation} \label{md}
dx:=\sum_{\mu(y)=\mu(x)+1}n(x,y)y,
\end{equation}
where $n(x,y)$ is given by
$$
n(x,y):=| \{u\in \widetilde{\M}(x,y) \} | \bmod{2}
$$
and extend it to general cochains in $\mathbb{Z}_2 {\langle \C(f) \rangle }$ by linearity.
\end{definition}

\begin{theorem}
$d^2=0$.
\end{theorem}

\begin{proof}
By definition $d^2x$ is given by
\begin{equation*}
d^2x = \sum_{\mu(z)=\mu(y)+1}n(y,z)\sum_{\mu(y)=\mu(x)+1}n(x,y)z.
\end{equation*}

The last statement in the Gluing Theorem shows that this equals summing over the boundary components of the compactification of $\widetilde{\M}(x,z)$:

\begin{gather*}
d^2x=\sum_z \sum_y  \big(| \{u\in \widetilde{\M}(x,y) \} |\cdot |\{v\in \widetilde{\M}(y,z)\}| \bmod{2}\big) z \\
=\sum_z \big(\sum_{(u,v)\in \partial \widetilde{\M}(x,z)}1\bmod{2}\big) z =0,
\end{gather*}
because every $1$-dimensional manifold without boundary is diffeomorphic either to $(0,1)$ or $S^1$ and therefore
the number of its boundary components is always zero modulo $2$.
\end{proof}

\begin{remark}
If one is working with $\mathbb{Z}$-coefficients, then every flow line $u \in \widetilde{\M}(x,y)$ is counted with a sign given by comparing the orientations of $\dot{u}$ and $\widetilde{\M}(x,y)$ (here the moduli spaces inherit induced orientations as transversal intersections of the orientable and coorientable submanifolds $W^u(x)$ and $W^s(y)$). In this case $d^2=0$ holds because the boundary components of $\widetilde{\M}(x,z)$ come with alternating signs.
\end{remark}

Finally we conclude that $\big( CM^*(f,g),d \big) $ is a cochain complex. That the homology of this complex is an invariant of $M$ can be seen either by relating it to another invariant of $M$, say singular cohomology (see Hutchings \cite{Hu}), or by showing that it is independent of the involved data $(f,g)$. The latter idea is worked out in detail in Weber \cite{We} using a ``continuation map'', a cochain map between two Morse complexes associated to different input data which induces a canonical isomorphism on cohomology:

Let $(f_0,g_0)$ and $(f_1,g_1)$ be two Morse-Smale pairs and let $(f_s,g_s)$ be a homotopy between them. Use a help function $h:[0,1]\to \mathbb{R}$ with critical set $\{0,1\}$ and $\mu(0)=0$, $\mu(1)=1$, to obtain a Morse function $F(s,p)=f_s(p) + h(s)$ on $[0,1]\times M$ (cf. Remark \ref{just}). Counting flow lines of $\nabla_{1\oplus g_s} F$ from $\{0\}\times \C(f_0)$ to $\{1\}\times \C(f_1)$ produces a cochain map $P(f_s):CM^*(f_0,g_0)\to CM^*(f_1,g_1)$ with the following properties:
\\

1. A generic homotopy of homotopies between $f_s$ and another homotopy $f_s'$ induces a cochain homotopy
\begin{gather}
H:CM^k(f_0,g_0)\to CM^{k-1}(f_1,g_1),  \\
dH + Hd = P(f_s) - P(f_s').
\end{gather}

2. If $(f_s',g_s')$ is a homotopy from $(f_1,g_1)$ to $(f_0,g_0)$, then $H(f_s')\circ H(f_s)$ is cochain homotopic to the identity.

3. If $(f_s,g_s)$ is the constant homotopy, then $H(f_s)$ is the identity on cochains.
\\

These properties imply, that $CM^*(f_0,g_0)$ and $CM^*(f_1,g_1)$ are canonically isomorphic (cf. Hutchings \cite{Hu}).

\subsection{Morse-Bott theory}
\label{mbt}

\subsubsection{General Morse-Bott theory}
\label{gmbt}

Morse-Bott theory is a generalization of Morse theory to functions $f$ where $Df$ is allowed to vanish along submanifolds of $M$ while the non-degeneracy condition still holds on their normal bundle. 
For the construction of the equivariant Morse complex we will need some Morse-Bott theory. Therefore we introduce it here quickly. For a detailed exposition we refer to Banyaga and Hurtubise \cite{BH}.
 
\begin{definition}
A smooth function $f: M \to \mathbb{R}$ is called \textit{Morse-Bott}, iff $\C(f)$ is a finite disjoint union of connected submanifolds of $M$, such that on the normal bundle of every $C \subset \C(f)$ the Hessian matrix of $f$ is non-degenerate.
\end{definition}

The non-degeneracy condition implies that the Morse index $\mu(x)$ for $x\in C$ is constant on a connected critical submanifold. So the Morse index $\mu(C)$ of a critical submanifold is well defined. 
\\

As in Morse theory there is a nice description of $f$ near critical submanifolds:

\begin{lemma}[Morse-Bott Lemma] \label{mbl}
Let $f,C$ be as above and $x\in C$: There exist local coordinates around $x$ and a local splitting of the normal bundle of $C$
\begin{equation*}
\mathcal{V}(C)=\mathcal{V}^u \oplus \mathcal{V}^s,
\end{equation*}
such that, if we identify $p\in M$ with $(v^0,v^u,v^s)$ in the local coordinate system, then $f$ is given by
\begin{equation*}
f(p) = f(v^0,v^u,v^s) = f(C) - \|v^u\|^2 + \|v^s\|^2.
\end{equation*}
\end{lemma}

As one might suspect, there are also generalizations of the other statements of the preceeding section, for example the ``Morse-Bott inequalities'', and there is a Morse-Bott complex computing the cohomology of $M$, see Banyaga and Hurtubise \cite{BH}. We end this subsection with one last important generalization of an aspect of Morse theory which we need in the following:

\begin{definition}    
The \textit{(un-)stable manifolds of a critical submanifold $C$} are defined as

\begin{equation*}
W^i(C):=\bigcup_{x\in C}W^i(x).
\end{equation*}
\end{definition}
\begin{prop} \label{usmbm}
$W^u(C)$ and $W^s(C)$ are smooth submanifolds of $M$ and their dimensions are given by 
\begin{gather*}
\dim W^u(C)=n - \mu(C),\\
\dim W^s(C)= \mu(C) + \dim C.
\end{gather*}
\end{prop}
Therefore, if $B,C \subset \C(f)$ and the intersection of the associated unstable and stable manifolds is transversal, then the \textit{connecting space of flow lines from $B$ to $C$} is defined as $\M(B,C):= W^u(B)\ti W^s(C)$ and 
\begin{equation*}
\dim \M(B,C) = \mu(C) - \mu(B) + \dim C.
\end{equation*}
Like in the Morse case, flow lines in $\M(B,C)$ are $\mathbb{R}$-shift invariant and the quotient is called the \textit{moduli space of flow lines from $B$ to $C$}:
\begin{equation*}
\widetilde{\M}(B,C):=\M(B,C) / \mathbb{R}.
\end{equation*}

\subsubsection{Flow lines with cascades}
\label{flwc}

Given a Morse-Bott function $f$ on $M$, there is a nice way of using the gradient flow line approach of Morse theory to compute $H^*(M)$. This idea of flow lines with cascades (in the following called FLWC) is due to Frauenfelder; for details, see \cite{Ff}.

Let $(h,g_0)$ be a Morse-Smale pair on $\C(f)$. For $x\in \C(h)$ a new Morse-like index of $x$ is defined as

\begin{equation*}
\lambda(x):=\mu_f(x)+\mu_h(x),
\end{equation*}
the sum of the Morse indices of $x$ with relation to $h$ and $f$.

\begin{definition}
Let $x,y \in \C(h)$. A \textit{flow line with $m$ cascades} from $x$ to $y$ is a tuple $(\underline{u},\underline{T})$ with
\begin{equation*}
\underline{u}=(u_1, \ldots, u_m),\ \underline{T}\in (\mathbb{R}_0^+)^{m-1}.
\end{equation*}
Here the $u_i \in C^{\infty}(\mathbb{R},M)$ are nonconstant solutions of 
\begin{equation*}
\dot{u}_i=\nabla_g f(u_i),
\end{equation*}
satisfying 
\begin{equation*}
\lim_{t \to -\infty}u_1(t)=p, \ \lim_{t \to \infty}u_m(t)=q
\end{equation*}
for some $p\in W^u(x,\nabla_{g_0}h)$ and $q\in W^s(y,\nabla_{g_0}h)$.

Furthermore for $i \in \{1,\ldots,m-1\}$ there are ordinary Morse flow lines $v_i \in C^{\infty}(\mathbb{R},\C(f))$ of $\nabla_{g_0}h$, such that
\begin{equation*}
\lim_{t \to -\infty}u_i(t)=v_i(0), \ \lim_{t \to \infty}u_{i+1}(t)=v_i(T_i).
\end{equation*}

\end{definition}

\begin{remark}
1. A flow line with zero cascades is just an ordinary Morse flow line on $\C(f)$.

2. The flow lines on $\C(f)$ are allowed to be constant, i.e. a cascade is allowed to converge to a critical point of $h$, but it will stay there only for a finite time interval. 
\end{remark}

\begin{definition}
The \textit{space of flow lines with $m$ cascades} from $x$ to $y$ in $\C(f)$ is denoted by 
\begin{equation*}
\mathcal{M}_m(x,y).
\end{equation*}
The group $\mathbb{R}^m$ acts freely on $\mathcal{M}_m(x,y)$ by timeshift on each cascade. The quotient is the \textit{moduli space of flow lines with $m$ cascades}:
\begin{equation*}
\widetilde{\mathcal{M}}_m(x,y).
\end{equation*}

\end{definition}

The usual transversality arguments show that these spaces are smooth manifolds with
\begin{align*}
\dim \widetilde{\M}_m(x,y)&=\lambda(y)-\lambda(x)-1 \\
&=\mu_f(y)+\mu_h(y)-\mu_f(x)-\mu_h(x)-1\\
&=\mu_h(y)-\mu_h(x)+m-1.
\end{align*}
Again, these moduli spaces admit natural compactifications and there is an associated gluing map, such that one is able to define a differential operator on $CC^*:=CC^*(f,h,g,g_0)=\mathbb{Z}_2\langle \C(h) \rangle$ (graded by $\lambda$):

\begin{definition}
For a generator $x\in \C(h)$ of $CC^*$ define
\begin{align*}
d^c x&:=\sum_m d_mx=\sum_{m} \sum_{\mu_h(y)=\mu_h(x)-m+1}n_m(x,y)y, \\
&n_m(x,y):= |\widetilde{\M}_m(x,y)| \bmod{2},
\end{align*}
and extend it to general cochains by linearity.
\end{definition}

Using continuation maps between FLWC-complexes associated to different functions and metrics, one proves the following theorem:
\begin{theorem}
The homology of the complex $(CC^*,d^c)$ is naturally isomorphic to $H^*(M)$.
\end{theorem}

\section{An equivariant Morse complex}
\label{ecm} 

In this chapter we introduce the $S^1$-equivariant Morse complex $(CM_{S^1},d_{S^1})$. As mentioned in the introduction $d_{S^1}$ counts flow lines ``jumping'' along orbits of the $S^1$-action. To model this jumping we use (higher) continuation maps of Morse homology. We set up the ``moduli spaces of $k$-jump flow lines'' which are the main topic of the next chapter.

\subsection{Definition of the equivariant Morse complex}
\label{defemc}

The $S^1$-equivariant Morse cochain groups are defined as follows: 

\begin{definition} 
Let $CM^*:=CM^*(f,g)=\mathbb{Z}_2 \langle \C(f) \rangle$ denote the Morse cochains associated to $(f,g)$, graded by their Morse index. Let $H^*(BG;\mathbb{Z}_2)=H^*_{S^1}(pt;\mathbb{Z}_2)=\mathbb{Z}_2[T]$ be the ring of polynomials in one variable with $\deg (T)=2$. We view elements $c\otimes p(T) \in CM^*\otimes \mathbb{Z}_2[T]$ as polynomials with coefficients in $CM^*$ and grade it by the sum of the Morse index of $c$ and twice the polynomial degree of $p$.
 
\begin{equation*}
CM_{S^1}^m:=\Big\{\sum_{i=0}^{j}a_iT^i | a_i \in CM^l,\ 2j+l=m\Big\} .
\end{equation*}
The differential 
\begin{equation*}
d_{S^1}: CM_{S^1}^* \to CM_{S^1}^{*+1}
\end{equation*}
is defined as 
\begin{equation} \label{deq}
d_{S^1} := d\otimes 1 + \sum_k R_{2k-1} \otimes T^k,
\end{equation}
where $d$ is the usual Morse differential. The operators $R_{2k-1}$ count $k$-jump flow lines of $\nabla_g f$ connecting critical points of index difference $2k-1$ - in contrast to $d$ they lower indices. Observe that the sum is finite, since for $2k-1>n=\dim M$ all $R_{2k-1}$ vanish, simply because there are no critical points with index greater than $n$ or with negative indices. 
\end{definition}

Of course this complex depends on input data like $f$ and $g$, but for notational convenience we omit this dependence and because the resulting complex will eventually turn out to be independent of these choices.

The remainder of this chapter is devoted to give a precise definition of the operators $R_{2k-1}$; mimicing the construction of the differential in ordinary Morse homology we now define $k$-jump flow lines and the associated moduli spaces.

\subsection{$k$-jump flow lines}
\label{ms}

We proceed in the following way: We start with the construction of $\widetilde{\M}_1(x,y)$, using a ``continuation map'' of Morse homology. Then we generalize these ideas to obtain $\widetilde{\M}_k(x,y)$ for $k\geq2$: To keep things simple, we carry out everything in full detail for $k=2$, the corresponding statements and proofs in the case $k> 2$ differ only by notational complexity since more dimensions are involved.

\subsubsection{$1$-jump flow lines}
\label{1jfl}

\begin{definition}
In the following we write 

\begin{equation*}
\sigma: S^1 \times M \to M, \ (s,p) \mapsto s.p:=\sigma(s,p)
\end{equation*}
for the action of $S^1$ on $M$ and $\sigma^p, \sigma_s$ for the maps 
\begin{equation*}
s \mapsto \sigma(s,p), \quad p \mapsto \sigma(s,p).
\end{equation*}
The defining properties for a $S^1$-action on $M$ translate to
\begin{align*}
\sigma_e=id_M \text{ and } \sigma_{s_2s_1}=\sigma_{s_2}\circ \sigma_{s_1},
\end{align*}
i.e. $ s \mapsto \sigma_s \in \text{Diff}(M) $ is a smooth group homomorphism.
\end{definition}

Let $x,y \in \C(f)$; for $s \in S^1$ choose a homotopy $f_t: M \times \mathbb{R} \to \mathbb{R} $ from $f$ to $f\circ \sigma_s $. The equation for a $1$-jump flow line then looks like this:
\begin{gather*}
\dot{u}(t)=\nabla f_t(u(t)), \\
u^-=x, 
u^+=s^{-1}.y.
\end{gather*}

In Morse theory this equation is known as a continuation equation. To avoid the time-dependency one translates this back into an autonomous equation on $M \times \mathbb{R}$ or $M \times S^1$ (to stay in the nice case of a closed manifold, cf. Remark \ref{just}). Since $s\in S^1$ works as a parameter we have to deal with a family of Morse functions smoothly parametrized by $S^1$, i.e a Morse-Bott function:

\begin{definition}\label{V_1}
For $(f,g)$ a Morse-Smale pair, let 
\begin{equation*}
\tilde{F}_1: M \times S^1 \times [0,1] \to \mathbb{R}, \quad
(p,s,r) \mapsto \tilde{F}_1(p,s,r)=\begin{cases}
    f(p) & \text{if $ r =0 $, }\\
    f(s.p) & \text{if $r=1$, }
  \end{cases} 
\end{equation*}
be a $S^1$-family of homotopies between $f$ and $f(s. \ \cdot \ )=f \circ \sigma_s$, such that $\tilde{F}_1$ is independent of $r$ for $r$ near $0$ and $1$.

Choose smooth $h: [0,1] \to \mathbb{R}$ with $h' \geq 0,\ h'^{-1}(0)=\{0,1\},\ h(0)=0$ and $h(1)> (\max{f} - \min f)$ (e.g. $h(r)=K(1+\sin \big(\pi(r-\frac{1}{2})\big) $ with $K>(\max f - \min f)$), such that $F_1: W_1 :=M \times S^1 \times [0,1] \to \mathbb{R}$ defined by 
\begin{equation*}
F_1(p,s,r):=\tilde{F}_1(p,s,r) + h(r)
\end{equation*}
has only critical points if $r=0,1$. 

Finally choose a product metric $G_1:=g\oplus1\oplus1$ on $W_1$.
\end{definition}

\begin{remark} \label{just}
Actually we should extend $F_1$ to be defined on $M\times S^1 \times S^1$ to continue working with a closed manifold. For this we would view the last $S^1$-factor as the interval $[-1,1]$ with endpoints identified, extend $\tilde{F}_1$ and $h$ symmetrically for $r < 0$ and restrict all following constructions to the subspace with $r \geq 0$. Keeping this in mind we continue working on $W_1 =M \times S^1 \times [0,1]$ for notational convenience and to emphasize the different roles of the $S^1$-factors involved - one is the jumping-parameter, while the other one parametrizes the homotopy. 
\end{remark}

\begin{lemma} \label{f1mb}
$F_1$ is a Morse-Bott function with 
\begin{equation*}
\C(F_1)= \bigcup_{x,y \in \C(f)} (A_x  \cup  B_y),
\end{equation*}
where 
\begin{align*}
A_x= & \{x\} \times S^1 \times \{0\}, \\
B_y= & \{(s^{-1}.y,s) | s\in S^1\} \times \{1\},
\end{align*} 
\end{lemma}

\begin{proof}

By construction critical points occur only at $r=0,1$, where $\tilde{F}_1\upharpoonright_{M \times S^1} $ is given by 
\begin{equation*}
(p,s) \mapsto f(p) \text{ and } f(s.p) \text{ respectively.}
\end{equation*}

At $r=0$:
\begin{align*}
D_p F_1(p,s,0) & =D f(p), \\
D_s F_1(p,s,0) & \equiv 0.
\end{align*}
Therefore 
\begin{equation*}
DF_1(p,s,0)=0 \text{ for } (p,s)\in \C(f)\times S^1.
\end{equation*}

At $r=1$:
\begin{align*}
D_pF_1(p,s,1) & =D_p (f\circ \sigma_s)(p) \\
 & = Df(\sigma_s(p))\cdot D\sigma_s(p), \\
D_sF_1(p,s,1) & = D_s(f\circ \sigma^p)(s) \\
& = Df(\sigma^p(s)) \cdot D\sigma^p(s).
\end{align*}

$\sigma_s$ is a diffeomorphism, so $D\sigma_s(p) \neq 0$ for all $p\in M$ and we conclude
\begin{equation*}
D F_1(p,s,1)=0 \text{ for } (p,s)\in \{(s^{-1}.x,s) | s\in S^1, x\in \C(f)\}. 
\end{equation*}

This also shows that the function $(p,r) \mapsto F_1(p,s,r)$ is Morse for every $s\in S^1$. So $F_1$ is a family of Morse functions, smoothly parametrized by $S^1$ ($\sigma$ is smooth) and therefore a Morse-Bott function. Note, that although the $S^1$-orbit of $x\in \C(f)$ might hit another critical point $y\in \C(f)$, the corresponding critical submanifolds of $F_1$ do not intersect in $W_1$:
\begin{gather*}
B_x \cap B_y \neq \emptyset \Longrightarrow (s^{-1}.x,s)=(\tilde{s}^{-1}.y,\tilde{s}) \\
\Longrightarrow s=\tilde{s} \\ \Longrightarrow x=y.
\end{gather*}
\end{proof}

We want to study gradient flow lines of $F_1$. Since $s\in S^1$ plays the role of a parameter, instead of the vector field $\nabla_{G_1}F_1$ we use a simpler, but still ``gradient-like'' vector field:

\begin{definition}
A \textit{gradient-like} vector field with respect to a Morse-Bott function $f$ defined on a smooth $n$-dimensional manifold $M$ is a smooth vector field $v$ satisfying:
\begin{equation*}
\mathcal{L}_v(f) > 0 \text{ on } M\setminus \C(f)
\end{equation*}
and around every point $c \in C \subseteq \C(f)$ there exist local coordinates $(p^i)$ such that
\begin{equation*}
v(p)=-2\sum_{i>\dim C}^{\mu(C)}{p^i\partial_{p^i}} + 2 \sum_{j> \mu(C)}^n{p^j\partial_{p^j}}.
\end{equation*}
This means, there is essentially no difference in studying the global structure of flow lines of $\nabla_g f$ and $v$.  
\end{definition}

\begin{lemma}
Let $V_1$ be the vector field on $W_1$ defined by
\begin{equation*}
(p,s,r) \mapsto \big(\nabla_g F_1(p,s,r),0,F_1'(p,s,r) \partial_r \big),
\end{equation*}
where $'$ denotes the derivative with respect to $r$. Then $V_1$ is a gradient-like vector field for $F_1$.
\end{lemma}
\begin{proof}
The Lie derivative of $F_1$ in direction $V_1$ at $(p,s,r)\not\in \C(F_1)$ is given by
\begin{gather*}
\mathcal{L}_{V_1}(F_1)(p,s,r)= DF_1(p,s,r)\cdot V_1(p,s,r) \\
= D_pF_1(p,s,r)\nabla_g F_1(p,s,r) + D_rF_1(p,s,r)\nabla_rF_1(p,s,r) \\
= g\big( \nabla_g F_1(p,s,r), \nabla_g F_1(p,s,r)\big) +  \big( \frac{\partial}{\partial r}F_1(p,s,r)\big)^2 > 0.
\end{gather*}

Fix $s\in S^1$. At $c\in C \subset \C(F_1)$  the restriction of $V_1$ to $M \times [0,1]$ is the gradient vector field of the Morse function $(p,r)\mapsto F_1(p,s,r)$. Thus, Lemma \ref{ml} implies the existence of suitable coordinates $(q^i)_{i\in\{1, \ldots,n+1\}}$ around $pr_{M\times[0,1]}(c)$, such that 
\begin{equation} \label{vm}
V_1\restriction_{M\times[0,1]}(q)= -2\sum_{i=1}^k q^i \partial_{q^i} + 2\sum_{j=k+1}^{n+1} q^j \partial_{q^j} \ \text{ for } k:=\mu(c).
\end{equation}   
Let $U$ be a neighbourhood of $c$ in $W_1$. Since $F_1$ is smooth, $C$ is a submanifold of $W_1$ and $\mu(c)$ is constant along $C$, the coordinate system $(q^i)_{i\in\{1, \ldots,n+1\}}$ depends smoothly on $s\in S^1$. Therefore, we can find coordinates 
\begin{equation*}
Q=(Q_1(p,s,r),\ldots,Q_n(p,s,r),s,Q_{n+1}(p,s,r))
\end{equation*}
such that $V_1$ is locally given by
\begin{equation*}
V_1(Q)= -2\sum_{i=1}^k Q^i \partial_{Q^i} + 2\sum_{j=k+1}^{n+1} Q^j \partial_{Q^j}.
\end{equation*}
\end{proof}

\begin{definition}
A \textit{$1$-jump flow line} between $x$ and $y$ is a solution of
\begin{align}
u: \mathbb{R} \to W_1, & \quad \dot{u}(t)=V_1(u(t)); \label{v1} \\
u^- \in A_x, & \quad u^+ \in B_y. \nonumber
\end{align} 
\end{definition}

Observe that the construction of $F_1$ (especially the right ``increasing behaviour'' of $h$, i.e. $h(1)>(\max{f}-\min f)$) implies that there are precisely three types of flow lines possible: 
\\

1. starting at $A_x$, ending in $A_y$.

2. starting at $A_x$, ending in $B_y$.

3. starting at $B_x$  ending in $B_y$. 
\\

Moreover, since $\tilde{F}_1$ is independent of $r$ near $0$ and $1$, the subsets 
$$
M\times S^1 \times \{0\}, M\times S^1\times \{1\} \subset W_1
$$
are both flow invariant. Because of $F_1(p,s,0)=f(p)$, flow lines of the first type are just $S^1$-families of ordinary Morse flow lines solving equation \eqref{eq1}.

\begin{lemma} \label{dimsf1}
The dimensions of the (un-)stable manifolds of $V_1$ are given by
\begin{align*}
&\dim W^u(A_x)= n-\mu(x)+2, \\
&\dim W^s(A_x)= \mu(x)+1, \\
&\dim W^u(B_x)= n-\mu(x)+1, \\
&\dim W^s(B_x)= \mu(x)+2.
\end{align*}
\end{lemma}

\begin{proof}
We have $\mu(A_x)=\mu(x)$, since $F_1$ is increasing along the $r$-direction and independent of $s \in S^1$ near $r=0$. At $B_x$ we have $h''(1)<0$, because $h$ attains its maximum at $r=1$, and therefore $\mu(B_x)=\mu(x)+1$. 
The dimension formulae follow now from Proposition \ref{usmbm}.
\end{proof}

\begin{definition}
For $x,y \in \C(f)$ we define
\begin{align*}
\M_1^A(x,y):= & W^u(A_x)\cap W^s(A_y), \\
\M_1^B(x,y):= & W^u(B_x)\cap W^s(B_y), \\
\M_1(x,y):= & W^u(A_x)\cap W^s(B_y).
\end{align*}
\end{definition}

\begin{prop} \label{ms1}
Given $x,y \in \C(f)$. $\M_1^A(x,y)$ and $\M_1^B(x,y)$ are equipped with a free $S^1$-action given for $\theta \in S^1$ by 
\begin{gather*}
\theta.u_A(t) =\theta.(p(t),s,0):=(p(t),\theta s,0), \\
\theta.u_B(t) =\theta.(p(t),s,1):=(\theta^{-1}.p(t),\theta s,1).
\end{gather*}
Moreover, all three spaces $\M_1^A(x,y), \M_1^B(x,y), \M_1(x,y)$ are equipped with a free $\mathbb{R}$-action given by
\begin{equation*}
\tau.u(t):=u(t+\tau),
\end{equation*}
and the corresponding quotients of spaces of flow lines at $r=0$ and $r=1$ are isomorphic:
\begin{equation*}
\M_1^B(x,y)/S^1 \cong \M_1^A(x,y)/S^1.
\end{equation*} 
\end{prop}

\begin{proof}

For $u \in  \M_1^A(x,y)$ we have $\theta.u \in \M_1^A(x,y)$ because $V_1$ does not depend on $s \in S^1$ at $r=0$.

For $u \in \M_1^B(x,y)$ we compute:
\begin{align*}
\frac{d}{dt}(\theta^{-1}.p)& =\frac{d}{dt}(\sigma_{\theta^{-1}}(p)) \\
& =D\sigma_{\theta^{-1}}(p) \cdot \dot{p} \\
& =D\sigma_{\theta^{-1}}(p) \cdot \Big(D\sigma_s(p)^* \cdot \nabla_g f(\sigma_s(p))\Big) \\
& =\big( D\sigma_s(p) \cdot D\sigma_{\theta^{-1}}(p)^* \big)^* \cdot \nabla_g f(\sigma_s(p)) \\
& =\Big(D\sigma_s(p) \cdot D\sigma_{\theta}\big(\sigma_{\theta^{-1}}(p)\big)\Big)^* \cdot \nabla_g f(\sigma_s(p)) \\
& =\Big (D\sigma_s \big( \sigma_{\theta}\big(\sigma_{\theta^{-1}}(p)\big) \big) \cdot D\sigma_{\theta} \big( \sigma_{\theta^{-1}}(p) \big) \Big)^*\cdot \nabla_g f(\sigma_s(p)) \\
& =\Big(D (\sigma_s \circ \sigma_{\theta}) \big( \sigma_{\theta^{-1}}(p) \big) \Big)^* \cdot \nabla_g f(\sigma_s(p)) \\
& = \big(D\sigma_{\theta s}(\theta^{-1}.p)\big)^*\cdot \nabla_g f( \sigma_{\theta s}(\theta^{-1}.p)) \\
& =\nabla_g F_1(\theta^{-1}.p,\theta s,1).
\end{align*}
Here we have used the chain rule and the fact, that for a diffeomorphism 
\begin{equation*}
\phi:M \to M
\end{equation*}
the pullback $\phi^*$ is given by the push-forward with $\phi^{-1}$:
\begin{gather*}
D\phi(p):T_pM \to T_{\phi(p)}M, \\
D\phi(p)^*:T_{\phi(p)}M \to T_pM = D\phi^{-1}(\phi(p)).
\end{gather*}

Clearly both actions are smooth (only smooth maps are involved) and free because they are free on the second factor.
\\

The second statement follows from the general fact that every solution $u: \mathbb{R} \to M $ of an autonomous dynamical system is $\mathbb{R}$-shift invariant and this shift clearly induces a free $\mathbb{R}$-action on the space of solutions.
\\

It remains to show the isomorphism between $\M_1^B(x,y)/S^1$ and $\M_1^A(x,y)/S^1$.

For $u=(p(t),s,0) \in \M_1^B(x,y)/S^1$ set $v:=(s.p(t),s,1)$, then
\begin{align*}
\frac{d}{dt}\big( s.p(t) \big) & = D\sigma_s(p(t))\cdot \dot{p}(t) \\
&=D\sigma_s(p(t)) \cdot \nabla_g \big(f\circ \sigma_s\big)(p(t)) \\
&=D\sigma_s(p(t)) \cdot \big(D\sigma_s(p(t)\big)^* \nabla_g f(\sigma_s(p(t))) \\
&=\nabla_g f(s.p(t)).
\end{align*}

Moreover, $(s.p)^-=s.p^-=s.(s^{-1}.x)=x$ and $s.p^+=s.s^{-1}.y$, so $v$ is in $\M_1^A(x,y)$. The same calculation for $u:=s^{-1}.v$ shows that this mapping is invertible and therefore an isomorphism.
\end{proof}

\begin{definition}
For $x,y \in \C(f)$ we define
\begin{gather*}
\widetilde{\M}_1(x,y):=\M_1(x,y) /\mathbb{R}\cong \M_1(x,y) \cap F_1^{-1}(a), \\
\widetilde{\M}_0(x,y):=\M_1^A(x,y)/ S^1 /\mathbb{R} \cong \M^A_1(x,y) /S^1 \cap F_1^{-1}(a),
\end{gather*}
where $a$ is any regular value between $F_1(A_x)$ and $F_1(B_y)$ or $F_1(A_y)$ respectively. The index $0$ in $\widetilde{\M}_0(x,y)$ emphasizes that this space is just $\widetilde{\M}(x,y)$, the moduli space of (non-jumping) flow lines used to construct the differential in ordinary Morse cohomology.
\end{definition}

\subsubsection{$2$-jump flow lines}
\label{2jfl}

Let $\Delta=\Delta^2$ denote the standard $2$-simplex as subset of $\mathbb{R}^2$ with vertices $(0,0)$, $(0,1)$ and $(1,1)$ and let $\vec{r}=(r_1,r_2)$ denote a point of $\Delta$. Similarly we write $\vec{s}$ for a point $(s_1,s_2)$ in the $2$-torus $T^2 \cong S^1\times S^1$.
\\

This time we use a $T^2$-family of homotopies, now parametrized by $\Delta$:

\begin{definition}[The vector field $V_2$]\label{V_2}
Choose $\tilde{F}_2: W_2 := M \times T^2 \times \Delta \to \mathbb{R}$ satisfying 

\begin{equation*}
\tilde{F}_2(p,\vec{s},\vec{r})=\begin{cases}
    f(p) & \text{if $ \| \vec{r} \| \leq \frac{1}{4} $ },\\
    f(s_1.p) & \text{if $\frac{7}{8} \leq \| \vec{r} \| \leq \frac{9}{8}$ }, \\
    f(s_2.s_1.p) & \text{if $ \frac{5}{4} \leq \| \vec{r} \| \leq \sqrt{2} $}. \\
  \end{cases}
\end{equation*}

Define $h_2: \Delta \to \mathbb{R}$ by $h_2(r_1,r_2):=h(r_1)+h(r_2)$ with $h$ from Section $\ref{1jfl}$. $h_2$ is smooth, strictly increasing with $\|\vec{r}\|$, $Dh_2(0,0)=Dh_2(0,1)=Dh_2(1,1)=0$ and $\mu(i,j)=i+j$. Furthermore $h_2$ satisfies the following: At $\partial \Delta$ the gradient of $h_2$ (with respect to the Euclidean metric) is always pointing along the boundary of $\Delta$. We set 

\begin{equation*}
F_2:=\tilde{F}_2+h_2 
\end{equation*}
and
\begin{equation*}
V_2:=(\nabla_g F_2,0,0,\nabla_{\vec{r}}F_2 ).
\end{equation*} 
\end{definition}

Recall that we work with the simplex $\Delta$ knowing that instead we could use the closed manifold $T^2$ to make everything precise without changing anything important. This is justified by the arguments in Remark \ref{just} and the tangent property of $\nabla_rh_2$ mentioned in the previous definition - here we would extend $h_2$ to $T^2$ (viewed as the unit square with opposing edges identified) by reflection on the diagonal $x=y$ and restrict attention to the subspace $x\leq y$. 

\begin{figure}[hbt]
	\centering
\includegraphics[width=0.50\textwidth]{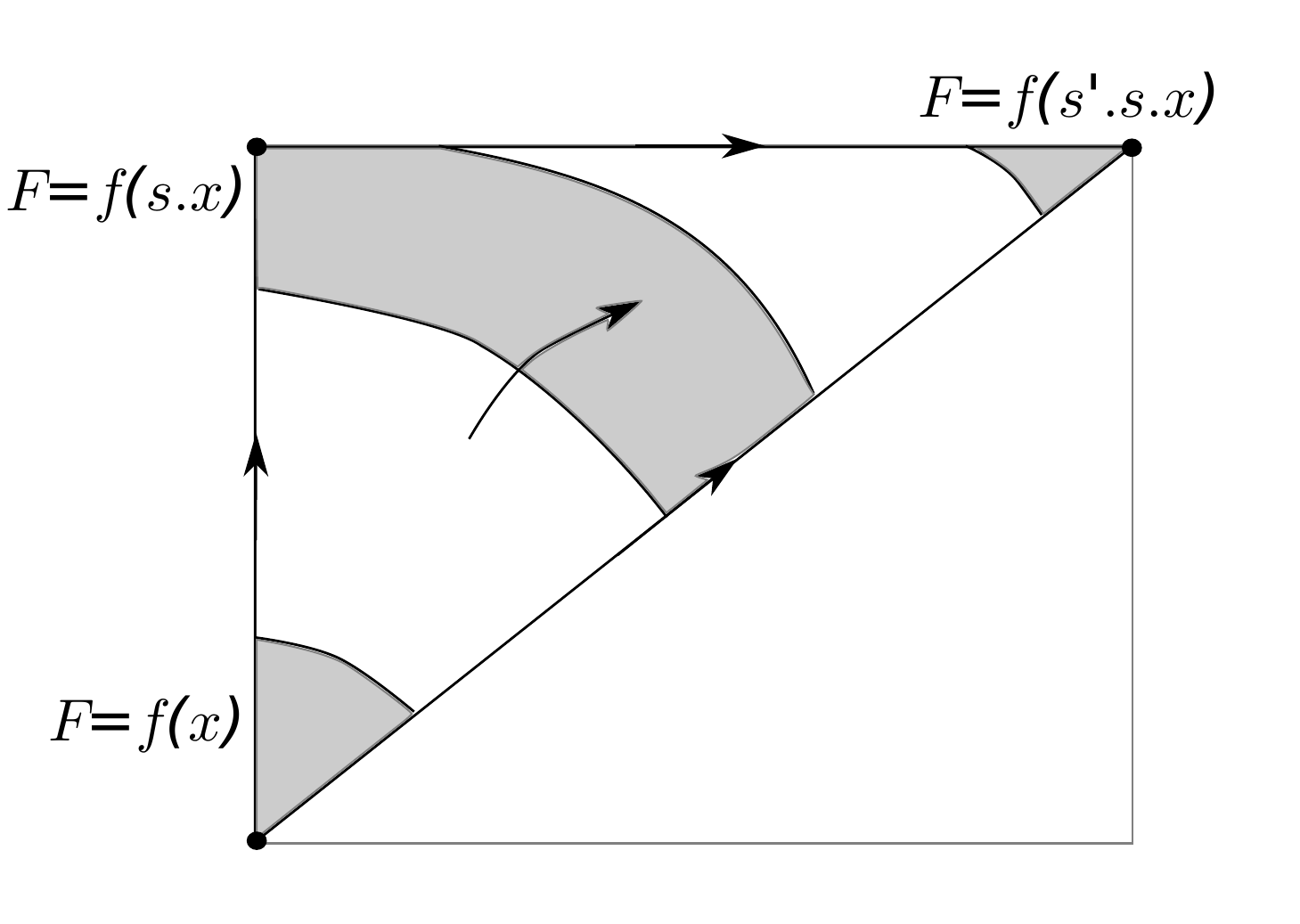}
	\caption{$\nabla h$ on $\Delta^2$}
	\label{bild6}
\end{figure}
 
\begin{lemma}
$F_2$ is Morse-Bott and, for $G_2=g \oplus 1^4$ a product metric on $W_2$, $V_2$ is gradient-like for $F_2$. Moreover, 
\begin{equation*}
\C(F_2)= \bigcup_{x,y,z \in \C(f)}(A_x \cup B_y \cup C_z),
\end{equation*}
where 
\begin{gather*}
A_x:= \{x\} \times T^2 \times \{(0,0)\}, \\
B_y:= \{ (s_1^{-1}.y,s_1) \in M \times S^1 \} \times S^1 \times \{(0,1)\}, \\
C_z:= \{ (s_1^{-1}.s_2^{-1}.z,s_1,s_2) \in M \times T^2 \} \times \{(1,1)\}. \\
\end{gather*}
The corresponding dimensions of the (un-)stable manifolds are given by
\begin{gather*}
\dim W^u(A_x)=n-\mu(x)+4, \\
\dim W^s(A_x)=\mu(x)+2, \\
\dim W^u(B_x)=n-\mu(x)+3, \\
\dim W^s(B_x)=\mu(x)+3, \\
\dim W^u(C_x)=n-\mu(x)+2, \\
\dim W^s(C_x)=\mu(x)+4.
\end{gather*}
\end{lemma}

\begin{proof}
This time $F_2$ is a smooth $T^2$ family of Morse functions on $M \times \Delta$. Using the same arguments as in Lemma \ref{f1mb} we see that $F_2$ is Morse-Bott and $V_2$ gradient-like. The structure of the critical submanifolds $A_x, B_x$ are derived in the same way, because 
\begin{align*}
F_2(p,s_1,s_2,0,0)&=f(p),\\
F_2(p,s_1,s_2,0,1)&=f(s_1.p).
\end{align*}
To calculate $DF_2(p,s_1,s_2,1,1)$ observe that 
\begin{align*}
\tilde{F}_2(p,s_1,s_2,1,1)& =f(s_2.s_1.p) \\
& =\big( f \circ \sigma_{s_2} \circ \sigma_{s_1} \big)(p) \\
& =\big( f \circ \sigma_{s_2s_1} \big)(p).
\end{align*}
Therefore, we can use the arguments of Lemma \ref{f1mb} again to conclude 
\begin{equation*}
C_x= \{ (s_1^{-1}.s_2^{-1}.x,s_1,s_2) \} \times \{(1,1)\}.
\end{equation*}

The dimensions of the (un-)stable manifolds are derived from Proposition \ref{usmbm} together with
\begin{equation*}
\dim W_2=n+4, \ \dim A_x=\dim B_x=\dim C_x=2,
\end{equation*}
and the fact that $h_2:\Delta \to \mathbb{R}$ ``adds'' $0,1,2$ to the Morse index of $A_x, B_x, C_x$ at the corresponding vertices $(0,0),(0,1),(1,1)$.  
\end{proof}

\begin{definition}
A \textit{$2$-jump flow line} from $x$ to $y$ is a solution of
\begin{align}
u: \mathbb{R} \to W_2, & \quad \dot{u}(t)=V_2(u(t)); \label{v2} \\
u^- \in A_x, & \quad u^+ \in C_y. \nonumber
\end{align} 
\end{definition}

Since we have choosen $h_2$ strictly increasing with $\|\vec{r}\|$ and scaled accordingly (cf. the condition $K>(\max f-\min f)$ in Definition \ref{V_1}) we can ensure that only flow lines of the three following types occur: 

The flow lines of $V_2$ live either in the subsets of $W_2$ containing the vertices of $\Delta$, travel from $A_x$ to $B_y$, from $B_y$ to $C_z$ (along the boundary of $\Delta$) or from $A_x$ to $C_z$. They correspond to $0$-,$1$- and $2$-jump flow lines respectively.

\begin{definition}
Regarding $F_2$ we set:
\begin{gather*}
\M^A_2(x,y):=W^u(A_x) \cap W^s(A_y), \\
\M^B_2(x,y):=W^u(B_x) \cap W^s(B_y), \\
\M^C_2(x,y):=W^u(C_x) \cap W^s(C_y), \\
\M^{AB}_2(x,y):=W^u(A_x) \cap W^s(B_y), \\
\M^{BC}_2(x,y):=W^u(B_x) \cap W^s(C_y), \\ 
\M_2(x,y):=W^u(A_x) \cap W^s(C_y).
\end{gather*}
\end{definition}

Observe that again on every space there is a $\mathbb{R}$-action by translation. In addition we have the following

\begin{prop}
The spaces $\M^A_2(x,y), \M^B_2(x,y)$ and $\M^C_2(x,y)$ are equipped with a free $T^2$-action. For $\vec{\tau}=(\tau_1,\tau_2) \in T^2$ the actions are given by
\begin{gather*}
u_A(t)=\big(p(t),s_1,s_2,0,0\big) \mapsto \big(p(t),\tau_1 s_1,\tau_2 s_2,0,0\big), \tag{1} \label{1} \\
u_B(t)=\big(p(t),s_1,s_2,0,1\big) \mapsto \big(\tau_1^{-1}.p(t),\tau_1 s_1,\tau_2 s_2,0,1\big), \tag{2} \label{2} \\
u_C(t)=\big(p(t),s_1,s_2,1,1\big) \mapsto \big(\tau_1^{-1}.\tau_2^{-1}.p(t),\tau_1 s_1,\tau_2 s_2,1,1\big) \tag{3} \label{3}.
\end{gather*}
The quotients are isomorphic to each other. Moreover $S^1$ acts freely on $\M^{AB}_2(x,y)$ and $\M^{BC}_2(x,y)$ by
\begin{gather*}
\theta.u_{AB}(t):=\big(p(t),s_1,\theta s_2,0,r_2(t)\big), \tag{4} \label{4} \\
\theta.u_{BC}(t):=\big(\theta^{-1}.p(t), \theta s_1, s_2,r_1(t),1\big) \tag{5} \label{5} ,
\end{gather*} 
and the quotients of both spaces are isomorphic to $\M_1(x,y)$.
\end{prop}

\begin{proof}
Obviously all actions are smooth (only smooth maps are involved) and free (on the $T^2$-factor).

The cases \eqref{1} and \eqref{4} follow directly from the fact that $F_2$ is independent of $\vec{s}$ at $\vec{r}=(0,0)$ and independent of $s_2$ at $\vec{r}\in \{0\}\times [0,1]$.

For \eqref{2} and \eqref{3} a straightforward generalization of the proof of Proposition \ref{ms1} shows that
\begin{align*}
& \frac{d}{dt}\big(\tau_1^{-1}.p_B(t)\big)=\nabla_g \big( f\circ \sigma_{\tau s_1} \big) (\tau_1^{-1}.p_B(t)), \\
& \frac{d}{dt}\big((\tau_1 \tau_2)^{-1}.p_C(t)\big)=\nabla_g \big( f\circ \sigma_{\tau_1 \tau_2 s_1 s_2}\big)\big((\tau_1 \tau_2)^{-1}.p_C(t) \big).
\end{align*}
The reason for this is that $F_1$ and $F_2$ have basically the same structure around the vertices of $[0,1]$ and $\Delta$.

The same argumentation works in the case $\eqref{5}$:
\begin{align*}
\frac{d}{dt}\big( \sigma_{\theta^{-1}}(p(t)) \big) & = D\sigma_{\theta^{-1}}(p(t)) \cdot \dot{p}(t) \\
& = D\sigma_{\theta^{-1}}(p(t)) \cdot \nabla_g F_2(p(t),s_1,s_2,r_1(t),1) \\
& = D\sigma_{\theta}(\theta^{-1}.p(t))^* \cdot \nabla_g F_2(p(t),s_1,s_2,r_1(t),1) \\
& = D\sigma_{\theta}(\theta^{-1}.p(t))^* \cdot \nabla_g F_2(\sigma_{\theta}(\theta^{-1}.p(t)),s_1,s_2,r_1(t),1) \\
& = \nabla_g \big( F_2 \circ \sigma_{\theta} \big) (\theta^{-1}.p(t),s_1,s_2,r_1(t),1) \\
& = \nabla_g  F_2 (\theta^{-1}.p(t),\theta s_1,s_2,r_1(t),1).
\end{align*}

Similiar to Proposition \ref{ms1} we define the isomorphism between the moduli spaces 
\begin{gather*}
\M_2^C(x,y)/T^2 \to \M_2^A(x,y)/T^2, \\
\M_2^B(x,y)/T^2 \to \M_2^A(x,y)/T^2,
\end{gather*}
by
\begin{align*}
&(p(t),s_1,s_2,1,1) \mapsto ((s_1s_2)^{-1}.p(t),s_1,s_2,0,0), \\
&(p(t),s_1,s_2,0,1)\  \mapsto (s_1^{-1}.p(t),s_1,s_2,0,0).
\end{align*}

The isomorphism between $\M^{BC}_2(x,y)/S^1$ and $\M^{AB}_2(x,y)/S^1$ is given by
\begin{equation*}
(p(t),s_1,s_2,r_1(t),1) \mapsto (s_2^{-1}.p(t),s_1s_2,s_2,0,r_1(t)).
\end{equation*}
Note, that this isomorphism is also a consequence of our construction, letting the subsets $M\times T^2 \times \{0\} \times [0,1],\ M\times T^2 \times [0,1] \times \{1\} \subset W_2$ be flow-invariant.

Finally, observe that $\M_2^A(x,y)/T^2 \cong \M_0(x,y)$, simply because both spaces consist of flow lines of $\nabla_gf$. 

Furthermore, $\M_2^{AB}(x,y)/S^1 \cong \M_1(x,y)$ via 

\begin{equation*}
\big(p(t),s_1,s_2,0,r_2(t)\big) \mapsto \big(p(t),s_1,r_2(t)\big).
\end{equation*}

\end{proof}

We conclude that the solutions of $\dot{u}=V_2(u)$ are either flow lines of equations $\eqref{eq1}$ or $\eqref{v1}$, or $2$-jump flow lines $\eqref{v2}$ travelling in $W_2$ from a critical submanifold at $(0,0)\in \Delta$ to a critical submanifold at $(1,1)\in \Delta$. 

\begin{definition}
For $x,y \in \C(f)$ the \textit{moduli space of $2$-jump flow lines} is defined as:
\begin{gather*}
\widetilde{\M}_2(x,y):=W^u(A_x) \cap W^s(C_y)/\mathbb{R} \cong \M_2(x,y) \cap F_2^{-1}(a),
\end{gather*}
where $a$ is any regular value between $(F_2(A_x),F_2(C_z))$.
\end{definition}

\subsubsection{$k$-jump flow lines for $k>2$}
\label{kjfl}

To obtain $\widetilde{\M}_k(x,y)$, proceed as above:
\\

Let $\Delta^k$ be the simplex obtained from $\Delta^{k-1}\subset \mathbb{R}^{k-1}\subset \mathbb{R}^k$ by adding the point $(1,\ldots,1)\in \mathbb{R}^k$ and connecting every vertex of $\Delta^{k-1}$ with it. Choose a $T^k$-family of ``homotopies'' $\tilde{F}_k: M\times T^k \times \Delta^k$ with 
\begin{equation*}
\tilde{F}_k(p,\vec{s},\vec{r})=\begin{cases}
    f(p) & \text{if $ \|\vec{r}\| < \delta $ },\\
    f(s_1.p) & \text{if $1-\delta \leq \|\vec{r}\| < 1+\delta$ }, \\
    f(s_2.s_1.p) & \text{if $ \sqrt{2}-\delta \leq \|\vec{r}\| < \sqrt{2}+\delta $}, \\
    \qquad \cdots & \\
    f(s_k. \cdots s_1.p) & \text{if  $\sqrt{k}-\epsilon \leq \|\vec{r}\|$ },
  \end{cases}
\end{equation*}
for some small $\delta > 0$.

Add $h_k:\Delta^k \to \mathbb{R}$, $h_k(r_1,\ldots,r_k):=\sum_{i=1}^k h(r_i)$. Again, $h_k$ is strictly increasing with $\|\vec{r}\|$, such that the gradient vector field $\nabla h_k$ (w.r.t. the Euclidean metric) is tangent to the boundary $\partial \Delta^k$ and the critical points of $h_k$ are the vertices $P\in \mathbb{R}^k$ of $\Delta^k$, satisfying $\mu(P)=\sum_{i=1}^{k}P_i$. 

Regarding the extension of $h_k$ to the $k$-cube (cf. Remark \ref{just}) we refer to Haiman \cite{Hai}: There is a subdivision of the $k$-cube into $k$-simplices, such that every vertex of a simplex is a vertex of the cube and the intersection of two simplices is a face of both of them - a very interesting problem is to find the minimum number of simplices one needs to do so. But for our purposes the sheer existence is enough to justify our definitions.
\\

This produces a Morse-Bott function $F_k: M\times T^k \times \Delta^k \to \mathbb{R}$, with $k$-dimensional critical submanifolds $A^0_x, A^1_x, \ldots, A^k_x$ for $x \in \C(f)$. Set 
\begin{equation*}
V_k(p,\vec{s},\vec{r}):=\big(\nabla_g F_k(p,\vec{s},\vec{r}),0,\ldots,0,\nabla_{\vec{r}}F_k(p,\vec{s},\vec{r})\big)
\end{equation*}
and define the \textit{moduli space of $k$-jump flow lines} from $x$ to $y$ as
\begin{equation*}
\widetilde{\M}_k(x,y):=W^u(A^0_x)\cap W^s(A^k_y)/\mathbb{R}.
\end{equation*}

Generalizing the above statements and proofs in the obvious manner we see that:
\\

1. The dimensions of the (un-)stable manifolds are given by
\begin{gather*}
\dim W^u(A^i_x)= n-\mu(x)+2k-i , \\
\dim W^s(A^i_x)= \mu(x)+k+i,\quad i=0,\ldots,k.
\end{gather*}
The crucial point here is that in a $\delta$-neighbourhood of the vertices of $\Delta^k$, where the critical submanifolds $A^i$ are situated, $F_k\upharpoonright_{M\times T^k}$ is always of the form 
\begin{equation*} 
(p,\vec{s}) \mapsto F(s_i. \cdots s_1.p).
\end{equation*}

2. Due to the construction, there are no flow lines possible from $A^i$ to $A^j$ for $i>j$.

3. The solutions of 
\begin{equation}
\dot{u}(t)=V_k\big(u(t)\big) \label{eqvk}
\end{equation}
are elements of (the corresponding quotients of the connecting spaces) $\M_0(x,y)$, $\M_1(x,y), \ldots$ or $\M_k(x,y)$ for some $x,y \in \C(f)$. This is because for $l<k,\ i+l<k$, the same argumentation as for $k=2$ shows that $\M^{A^iA^{i+l}}_l(x,y)/T^l$ defined by $F_k$ is isomorphic to $\M_l(x,y)$ defined by $F_l$. Note that the corresponding subsets of $W_k$ containing $l$-jump flow lines are flow-invariant.

\subsection{The operators $R_{2k-1}$}
\label{rks}

So far we have set up the moduli spaces of $k$-jump flow lines consisting of (equivalence classes of) solutions of \eqref{eqvk}.

Due to the construction of $V_k$ every $u \in \M_k(x,y)$ starts at a critical submanifold $A^0_x=\{x\} \times (S^1)^k \times \Delta^k \subset W_k$ and flows to a critical submanifold $A^k_y \subset W_k$, with its projection to $M$ being the $S^1$-orbit of $y$. In the next chapter we will see that given $x$ and $y$ in $\C(f)$ with $\mu(x)-\mu(y)=2k-1$ there are only finitely many such $u$ in $\M_k(x,y)$ and therefore we are able to make the following

\begin{definition}[The operators $R_{2k-1}$]
For a generator $x\in \C(f)$ of $CM^*$ we set
\begin{equation} \label{rk}
R_{2k-1}x:= \sum_{ \mu(y)=\mu(x)-(2k-1) }n_k(x,y) y,
\end{equation}
where $n_k(x,y):=|\widetilde{\M}_k(x,y)|\bmod{2}$, and extend it to general cochains by linearity.
\end{definition}

\section{The moduli spaces $\tilde{\M}_k(x,y)$}
\label{mods}

In this chapter we take a closer look at the moduli spaces defined in Section \ref{ms}. We state and proof according transversality, compactness and gluing theorems which justify the definition of $R_{2k-1}$ given in Section \ref{rks} and are needed for showing $d_{S^1}^2=0$. The proofs of compactness and gluing, using the theory of dynamical systems, follow closely Weber \cite{We}, where these theorems are proven for a Morse-Smale pair $(f,g)$. Throughout this chapter we omit the subscripts $k$ in $F,V,\ldots$, except for the moduli spaces $\widetilde{\M}_k(x,y)$. Moreover, in all following constructions for $\M_k(x,y)$, if $l < k$, we view $\M_l(x,y)$ as a subset of $W_k$.

\subsection{Transversality}
\label{tv}

In this section we show, that by altering the vector field $V$ it is possible to achieve transversality of the intersections of all unstable and stable manifolds of the critical submanifolds of $F$. We will change the vector field $V$ only on small open sets containing no critical submanifolds, so that the global structure of the flow is not harmed. In other words, the new vector field will still be gradient-like. 
\\

In the following we have to keep track of the vector fields to which we assign the (un-)stable manifolds. Therefore $W^i(C,v)$ will denote the (un-)stable manifold of the critical submanifold $C$ associated to the vector field $v$.

\begin{theorem}[Transversality]
There is a gradient-like vector field $V'$ (in the following sections again denotet by $V$) for $F$, differing from $V$ only on small open sets outside critical submanifolds, such that all unstable and stable manifolds intersect transversally.
\end{theorem}

Following the ideas of Nicolaescu \cite{N} in the proof of a similiar statement we need two lemmata:

\begin{lemma}
Given a smooth function $f:M \to \mathbb{R}$ and $v$ a gradient-like vector field for $f$, let $[a,b]$ consist only of regular values of $f$ and for any interval $I$ of regular values set $M_I:=f^{-1}(I)$. 

If $h:M_b \to M_b$ is a diffeomorphism isotopic to the identity, then there exists a gradient-like vector field $w$ which equals $v$ outside $M_{(a,b)}=f^{-1}((a,b))$ such that:
\begin{equation*}
h \circ \phi^v_{b,a} = \phi^w_{b,a}.
\end{equation*}
\end{lemma}

Here $\phi^v_{b,a}: M_a\to M_b$ is the ($t=1$)-flow of the normalized vector field $\langle v\rangle:= \frac{1}{\mathcal{L}_v f}\cdot v$. If there are no critical values in $[a,b]$ this is a diffeomorphism with inverse $\phi^v_{a,b}$ given by the backward-flow.

\begin{proof}
Without loss of generality assume $a=0, b=1$ and that the isotopy $h_t$ is independent of $t$ near $0$ and $1$.

Define a diffeomorphism 
\begin{equation*}
\psi : [0,1] \times M_1 \to M_{[0,1]}, \quad (t,p) \mapsto \phi^v_{t,1}(p).
\end{equation*}
$\psi^{-1}$ is given by 
\begin{equation*}
y \mapsto \big(f(y),\phi^v_{1,f(y)}(y)\big).
\end{equation*}

Regarding the isotopy $h_t$ there is another diffeomorphism 
\begin{equation*}
H : [0,1] \times M_1 \to [0,1] \times M_1, \quad (t,p) \mapsto (t,h_t(p)).
\end{equation*}

Now let $\tilde{w}$ be the pushforward of $\langle v \rangle$ by $ \psi \circ h_t \circ \psi^{-1}$. Rescale it, 
\begin{equation*}
w:=(\mathcal{L}_v f) \tilde{w},
\end{equation*}
and extend it to a vector field on $M$ coinciding with $v$ outside $M_{[0,1]}$. Then $\langle w \rangle=\tilde{w}$ and therefore 
\begin{equation*}
h \circ \phi^v_{1,0} = \phi^w_{1,0}.
\end{equation*}
\end{proof}

\begin{lemma}
Let $X,Y$ be submanifolds of $M$ and assume that $X$ has a tubular neighbourhood. Then there exists a diffeomorphism $h: M \to M$ smoothly isotopic to the identity such that $h(X) \ti Y$.
\end{lemma}

\begin{proof}
The proof is from Milnor \cite{Mi3}. 

Let $m$ be the dimension of the normal bundle of $X$ in $M$ and let
\begin{gather*}
k:X \times \mathbb{R}^m \to U \subset M, \\
k\restriction_{X\times \{0\}}=id_X
\end{gather*}
be a tubular neighbourhood of $X$. Set $Y_0:=U \cap Y$ and $g:= pr_2 \circ (k^{-1})\restriction_{Y_0}$, where $pr_2$ is the projection onto $\mathbb{R}^m$.

For $x\in \mathbb{R}^m ,\ k(X\times \{x\})$ does not intersect $Y$ transversally iff there is $\xi\in Y_0$ with $x=g(\xi)$ and $Dg(\xi)$ has not maximal rank. Since $k$ is a diffeomorphism and $\operatorname{rank} (pr_2)$ is constant, this means $\xi \in \C(g)$. Therefore, Sard's theorem (cf. Milnor and Stasheff \cite{Mi1}) implies that for $x$ outside a set of measure zero the intersection will be transversal.

We construct the isotopy $h_t$ as follows:

Choose $x \in \mathbb{R}^m \setminus g\big(\C(g)\big)$ and a smooth vector field $v$ on $\mathbb{R}^m$, satisfying:
\begin{equation*}
v(y)=\begin{cases}
    x & \text{if $ \| y \| \leq \| x \| ,$ }\\
    0 & \text{if $ \| y \| \geq 2\| x \| .$ }
  \end{cases}
\end{equation*}
The flow $\phi_t$ of $v$ is defined for all $t \in \mathbb{R}$, $\phi_0=id$ and $\phi_1$ maps $0$ to $x$. 

Finally we define $h_t : M \to M$ by
\begin{equation*}
h_t(p)=\begin{cases}
    k\big(q,\phi_t(y)\big) & \text{if $p=k(q,y) \in U ,$ }\\
    p & \text{if $ p \not\in U .$ }
  \end{cases}
\end{equation*}
Clearly, $h_t$ is a diffeomorphism for all $t\in [0,1]$, depending smoothly on $t$, satisfying $h_0=id$ and $h_1(X)\ti Y.$

\end{proof}

Now we are able to prove the Transversality Theorem:

\begin{proof} 
In the last chapter we have seen that \begin{equation*}
\C(F)=\C(F_k)=A^0 \cup A^1 \cup \cdots \cup A^k
\end{equation*} 
with $A^i:=\bigcup_{x\in \C(f)}A^i_x$ and flow lines from $A^i$ to $A^j$ only occur if $i \leq j$. More precisely, we have constructed $F$ such that, for $x,y \in \C(f)$, the inequality $F(A^i_x)<F(A^j_y)$ implies $i<j$ or $i=j$ and $\mu(x)< \mu(y)$.

Next, observe that in general 
\begin{equation*}
W^u(A^i) \ti W^s(A^j) \Longleftrightarrow W^u(A^i)_c \ti W^s(A^j)_c \quad \forall i\leq j, 
\end{equation*}
where $W^i(A)_c:=W^i(A)\cap F^{-1}(c)$ for $c$ a regular value. 
\\

Since we have assumed that $(f,g)$ is Morse-Smale, we have
\begin{align*}
&W^u\big( (x,\vec{s},\vec{0})\in A^0_x,V \big)  = W^u(x,\nabla_g f)\times T^k \times \{(0,\ldots,0)\} \\
\ti \ &W^s(y,\nabla_g f)\times T^k \times \{(0,\ldots,0)\} = W^s\big( (y,\vec{s},\vec{0}) \in A^0_y,V\big), \\
& \quad \Longrightarrow  W^u(A^0_x,V)  \ti W^s(A^0_y,V) \quad \forall x,y \in \C(f) .
\end{align*}

Now start with $y\in \C_0(f)$, the minimum of $F$, choose a regular value 
\begin{equation*}
\displaystyle{ c \in \big(\max_{x\in \C(f)}F(A^0_x),F(A^1_y)\big) } 
\end{equation*}
and $\epsilon > 0$ so small that $(c-2\epsilon,c)$ contains no critical values. 

Set $a:=c-2\epsilon,\ b:=c-\epsilon$. Then $X_b:=\bigcup_{x\in \C(f)}W^u(A^0_x,V)_b$ and $W^s(A^1_y,V)_b$ are smooth $T^k$-families of spheres embedded in $W_b$. 

Therefore, they are smooth submanifolds (with boundary) of $W_b$ admitting tubular neighbourhoods and we can use Lemma $4.4$ to obtain $h: W_b \to W_b$ isotopic to the identity, making the intersection of $X_b$ and $W^s(A^1_y,V)_b$ transversal. 

Lemma $4.3$ asserts the existence of a gradient-like vector field $V'$, equal to $V$ outside $W_{(a,b)}$, with 
\begin{equation*}
\phi^{V'}_{b,a}=h\circ \phi^{V}_{a,b}.
\end{equation*}
We have
\begin{gather*}
W^u(A^0_x,V')_a=W^u(A^0_x,V)_a \ \forall x \in \C(f), \\
W^s(A^1_y,V')_b=W^s(A^1_y,V)_b, 
\end{gather*}
and therefore for all $x\in \C(f)$:
\begin{align*}
W^u(A^0_x,V')_b & =\phi^{V'}_{b,a}\big(W^u(A^0_x,V')_a\big) \\
& =h\circ \phi^{V}_{a,b}\big( W^u(A^0_x,V')_a \big) \\
& =h\circ \phi^{V}_{a,b}\big( W^u(A^0_x,V)_a \big) \\
& =h\big( W^u(A^0_x,V)_b \big)\\
&\ \ti W^s(A^1_y,V)_b=W^s(A^1_y,V')_b.
\end{align*}

Repeating these steps, first for all $y\in \C(f)$ with $X_b$ uniting all unstable manifolds 
\begin{equation*}
\bigcup_{x\in \C(f)}W^u(A^0_x,V)_b\ \text{ and} \bigcup_{\mu(z)<\mu(y)}W^u(A^1_z,V)_b,
\end{equation*}
then continuing successively with $A^2, \ldots, A^k $, in every step choosing the values $a,b$ such that $V$ has not been altered on $W_{(a,b)}$ previously. After a finite number of steps ($\C(f)$ is finite, $\C(F)$ is a finite union of submanifolds of $W$), we have achieved transversality of all intersections of unstable and stable manifolds.
\end{proof}

Transversality of the intersection of $W^u(A^i_x,V)$ and $W^s(A^j_y,V)$ implies the following dimension formulae:
\begin{cor} \label{dims}
$\M(A^i_x,A^j_y)$ and $\widetilde{\M}(A^i_x,A^j_y)$ are smooth finite-dimensional submanifolds of $W$ with
\begin{align*}
\dim \M(A^i_x,A^j_y)&= \dim\big( W^u(A^i_x,V)\ti W^s(A^j_y,V) \big) \\
&= \mu(y)-\mu(x)+k+j-i,\\
\dim \widetilde{\M}(A^i_x,A^j_y)&= \mu(y)-\mu(x)+k+j-i -1.
\end{align*}
In particular, 
\begin{equation*}
\dim \widetilde{\M}_k(x,y)=\mu(y) - \mu(x) + 2k -1.
\end{equation*}

\end{cor}
\newpage

\subsection{Compactness}
\label{comp}

We continue to examine the properties of the moduli spaces by showing that $\widetilde{\M}_k(x,y)$ can be compactified using so-called ``broken flow lines'' $(u_1,\ldots,u_m)$. These broken flow lines are defined as in Definition \ref{broken}, except that now they are allowed to consist of elements of different moduli spaces $\M_l(x',y')$ with $0\leq l \leq k$.

\begin{definition}
A \textit{broken flow line of type $(m;\Gamma)$} is a $m$-tuple $(u_1,\ldots,u_m)$, such that
\begin{gather*}
u_i \in \M_{\Gamma_i}(x_i,x_{i+1}), \\ 
 x_i \in \C(f)\ (i=1,\ldots,m), \\
 \Gamma \in \mathbb{N}^m: \Gamma_i \in \{0,1,\ldots,k\}. 
\end{gather*}
A subset $K \subset \widetilde{\M}_k(x,y)$ is called \textit{compact up to broken flow lines of type $(m;\Gamma)$}, if for every sequence $(w_j)_{j\in\mathbb{N}}\subset K$ there exists a broken flow line $(u_1,\ldots,u_m)$ of type $(m;\Gamma)$ with
\begin{equation*}
O(w_j) \longrightarrow (u_1,\ldots,u_m) \text{ for } j \to \infty.
\end{equation*}
Here convergence is meant with relation to the Riemannian distance $d$ on $W_k$ (again secretly identifying $W_k=M\times T^k \times \Delta^k$ with the closed manifold $M\times T^k\times T^k$ as discussed in Section $3.2$). This implies $u^+_i=u^-_{i+1}$, i.e. a broken flow line which is limit of a sequence in $\widetilde{\M}_k(x,y)$ does not ``jump'' along critical submanifolds, although it will in general consist of jumping parts (cf. Fig. \ref{bild2}).
\end{definition}

Recall that $\widetilde{\M}_0(x,y)$ is naturally isomorphic to the moduli space $\widetilde{\M}(x,y)$ of ordinary flow lines associated to $(f,g)$, which we assumed to be a Morse-Smale pair. Therefore, Theorem \ref{mcomp} already tells us how to compactify these spaces.

In the following we restrict attention to the case $\dim \widetilde{\M}_k(x,y)=0,1$, because higher-dimensional moduli spaces are irrelevant for our purposes.

\begin{figure}[hbt]
	\centering
\includegraphics[width=0.90\textwidth]{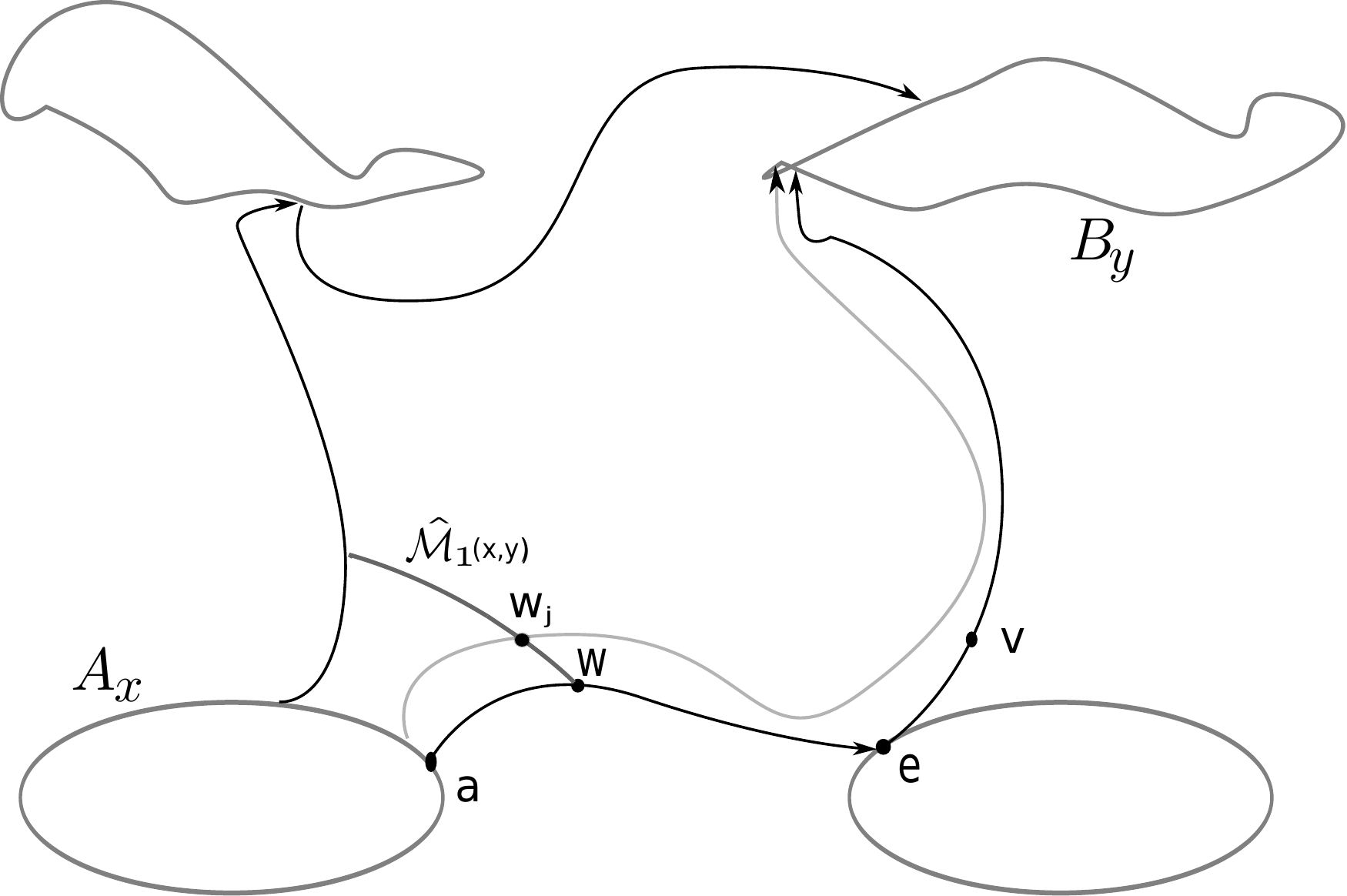}
	\caption{Broken flow lines for $k=1$}
	\label{bild2}
\end{figure}

\begin{theorem}
The moduli spaces of $k$-jump flow lines $\widetilde{\M}_k(x,y)$ are compact up to broken flow lines of type  $(m;\Gamma)$, where
\begin{align*}
&m=1,\Gamma=k, \text{ that is $\widetilde{\M}_k(x,y)$ is already compact, or } \\
&m=2,\Gamma \in \{(i,j)| i + j=k \text{ and } i,j \in \{0,1,\ldots,k\}\}.  
\end{align*}
\end{theorem}

\begin{proof}
We adapt the proof from Weber \cite{We} to our situation: For a regular value $c\in (F(A^0_x),F(A^k_y))$ set 
\begin{equation*}
\widetilde{\M}_k(x,y):=W^u(A^0_x)\cap W^s(A^k_y) \cap F^{-1}(c)
\end{equation*}
and let $(w_j)_{j\in \mathbb{N}}$ be a sequence in $\widetilde{\M}_k(x,y)$. 

Then there is a subsequence converging to some element $w$ of the compact set $F^{-1}(c)$. We denote the subsequence again by $(w_j)$, as we will always do in the following when choosing a subsequence. 

Clearly, $O(w) \in \M_l(z,z')$ for some $z,z' \in \C(f)$ and $l\in\{0,1,\ldots,k\}$, because $V$ is gradient-like. Let $\Phi_t$ denote the flow of $V$; this map is continuous and therefore 
\begin{equation*}
O(w)=\bigcup_{t\in \mathbb{R}}\Phi_t(w) \subset \overline{\M_k(x,y)}
\end{equation*}
is a subset of the closure of $\M_k(x,y)$. Using continuity again we conclude 
\begin{equation*}
a:=\lim_{t\to -\infty}\Phi_t(w) \text{ and } e:=\lim_{t\to \infty}\Phi_t(w) \in \overline{\M_k(x,y)}
\end{equation*}
are also elements of the closure of $\M_k(x,y)$.

\textbf{Claim:} If $z' \neq y$, then there is a flow line in $\M_{l'}(z',y)$ for some $l'\in \{0,1,\ldots,k\}$. In other words:
\begin{equation*}
z'\neq y \Longrightarrow  \ \exists v \in W^u(e) \cap \overline{\M_k(x,y)} \text{ and } v\neq z'.
\end{equation*}

\textbf{Proof:} (by contradiction)
At $e$ we have a splitting of the tangent space (cf. Lemma \ref{mbl})

\begin{equation*}
T_eW= E^0 \oplus E^u \oplus E^s.
\end{equation*}

The Hartman-Grobman theorem for non-hyperbolic critical points (see Kirchgraber and Palmer \cite{KP}) asserts that there exists a homeomorphism $h: U \subset W \to N \subset T_eW$ with $h(e)=0$, such that 

\begin{equation*}
h\big(\Phi_t(p)\big) = \mathcal{A}_t\big(h(p)\big),
\end{equation*}
where $\mathcal{A}_t$ is the flow of the linearization $DV(e)$ of $V$ on $T_eW$. The restriction of $\mathcal{A}_t$ to $E^u \oplus E^s$ is the linear map $A_t$ of Remark \ref{tangent} and $h$ identifies a neighbourhood of $e$ in $W^i(e)$ with a neighbourhood of zero in $E^i$, where $\mathcal{A}_t$ acts strictly expanding or contracting respectively.

Assume that $w_j,w \in U$ (otherwise let them flow with $\Phi_t$ sufficiently long and choose a subsequence) and identify in the following $w_j$ and $w$ with their images under $h$.

Now suppose the above statement is not true. Since the flows $\Phi_t$ and $\mathcal{A}_t$ are conjugate we are able to transfer the problem to $T_eW$. This means, that for every $\epsilon > 0$ there exists $\delta > 0$, such that $S \subset E^u$, a sphere of radius $\epsilon$, admits a $\delta$-neighbourhood $B \subset T_eW$ with $B\cap \M_k(x,y)=\emptyset$.

Choose $S,B \subset N$ and assume $\|w\| < \frac{\delta}{2}$ (otherwise use again the flow and choose a subsequence).
Now $\mathcal{A}_t$ is linear on $E^u \oplus E^s$, so we can write it in the following form 
\begin{equation*}
\mathcal{A}_t(q)\restriction_{E^u \oplus E^s} = (A^u_tq^u, A^s_tq^s)
\end{equation*}
for $q=(q^u, q^s) \in E^u \oplus E^s \cap N$ and $A^u_t\oplus A^s_t=A_t$.
Since $A^u_t$ is a strict dilatation for $t> 0$ and $0$ is its only fixed point, there exists $t_0 > 0$ with $\|A^u_tw^u_j\| > \epsilon$. On the other hand $A^s_t$ is a strict contraction for $t> 0$. Therefore, for $j$ large enough, such that $\|w_j\| < \delta$, we have $\|A^s_tw^s_j\| < \delta$ for positive $t$.
But this implies that the flow line $O(w_j)$ through $w_j$ hits $B$, contradicting our assumption.

Moreover, this shows that locally near critical points in critical submanifolds the flow lines through $w_j$ converge uniformly to the broken flow line $\big(O(w),O(v)\big)$.
Outside neighbourhoods of critical submanifolds the flow lines through $w_j$ converge to $O(w)$ on compact time intervals $J\subset \mathbb{R}$, because the map
\begin{gather*}
W\times W\times J \to \mathbb{R}, \\
(p,p',t) \mapsto d(\Phi_t(p),\Phi_t(p'))
\end{gather*}
is smooth and therefore Lipschitz continuous since $W\times W\times J$ is compact. This means
\begin{equation*}
d(\Phi_t(p),\Phi_t(p')) \leq C\cdot d(p,p'),
\end{equation*}
for all $(p,p',t) \in W\times W\times J$ and a constant $C=C(W,V,J)>0$,
implying the asserted convergence.

To show convergence to $O(v)$ replace the regular value $c$ with $c'=F(v)$ and $w_j$ with $\widetilde{w}_j=O(w_j)\cap F^{-1}(c')$ and argue as above.
\\

Now repeat everything with respect to the backward flow $\Phi_{-t}$ and $a=O(w)^-$, the starting point of $O(w)$, to see that if $z \neq x$, then there is $v' \in W^s(a) \cap \overline{\M_k(x,y)}, \ v'\neq z$ and we have uniform convergence of $O(w_j)$ to $(O(v'),O(w))$ on compact time intervals.
\\

Finally, after a finite number of steps ($\C(f)$ is finite, $\C(F)$ is a finite collection of submanifolds of $W$), we end up with a broken flow line of type $(m,\Gamma)$, starting at $A^0_x$ and ending in $A^k_y$ in its most general form
\begin{equation} \label{bfl}
(u^0_1,\ldots,u^0_{n_0},v_1,u^1_1,\ldots,u^1_{n_1},v_2,u^2_1,\ldots,u^{m_b-1},v_{m_b},u^{m_b}_1,\ldots,u^{m_b}_{n_{m_b}}).
\end{equation}
Here the $u$'s are $0$-jump flow lines, the $v_i$ are elements of $\M_{l_i}(z_{l_i},z_{l_i}')$ with $i=1,\ldots,m_b \ (1 \leq m_b,l_i \leq k),$ and 

\begin{equation*}
m=m_b+m_0=m_b+n_0 +\ldots+n_{m_b}. 
\end{equation*}
Here $m_b\leq k$, simply because $\Delta^k$ has only $k+1$ vertices and $V$ is gradient-like.
\\

To finish the proof we must show that only those types mentioned in the theorem can occur. During the next paragraph let $m_0=0$, i.e. $m=m_b$ - in the following ordinary flow lines are irrelevant.
\\

That there are broken flow lines of type $(m,\Gamma)$ with $m>2$ is due to our construction using the simplices $\Delta^k$ and situating the critical submanifolds at their vertices - for $k>2$ we get a little bit more than we need: 

In $\Delta^2$ two edges correspond to $1$-jump flow lines and one to $2$-jump flow lines; here a $2$-jump flow line can break up into two $1$-jump flow lines and this is exactly how we want it to be. But in $\Delta^3$ we have one redundant edge (corresponding to $1$-jump flow lines from $f(s_1.p)$ to $f(s_2.s_1.p)$) which makes a broken flow line of the type $(3;(1,1,1))$ possible (cf. Fig. \ref{bild1}). Likewise for $k=4$ we encounter broken flow lines of type $(4;(1,1,1,1)),\ (3;(1,1,2)),\ (3;(1,2,1))$ and $(3;(2,1,1))$ and so on for $k>4$.

\begin{figure}[hbt]
	\centering
\includegraphics[width=0.60\textwidth]{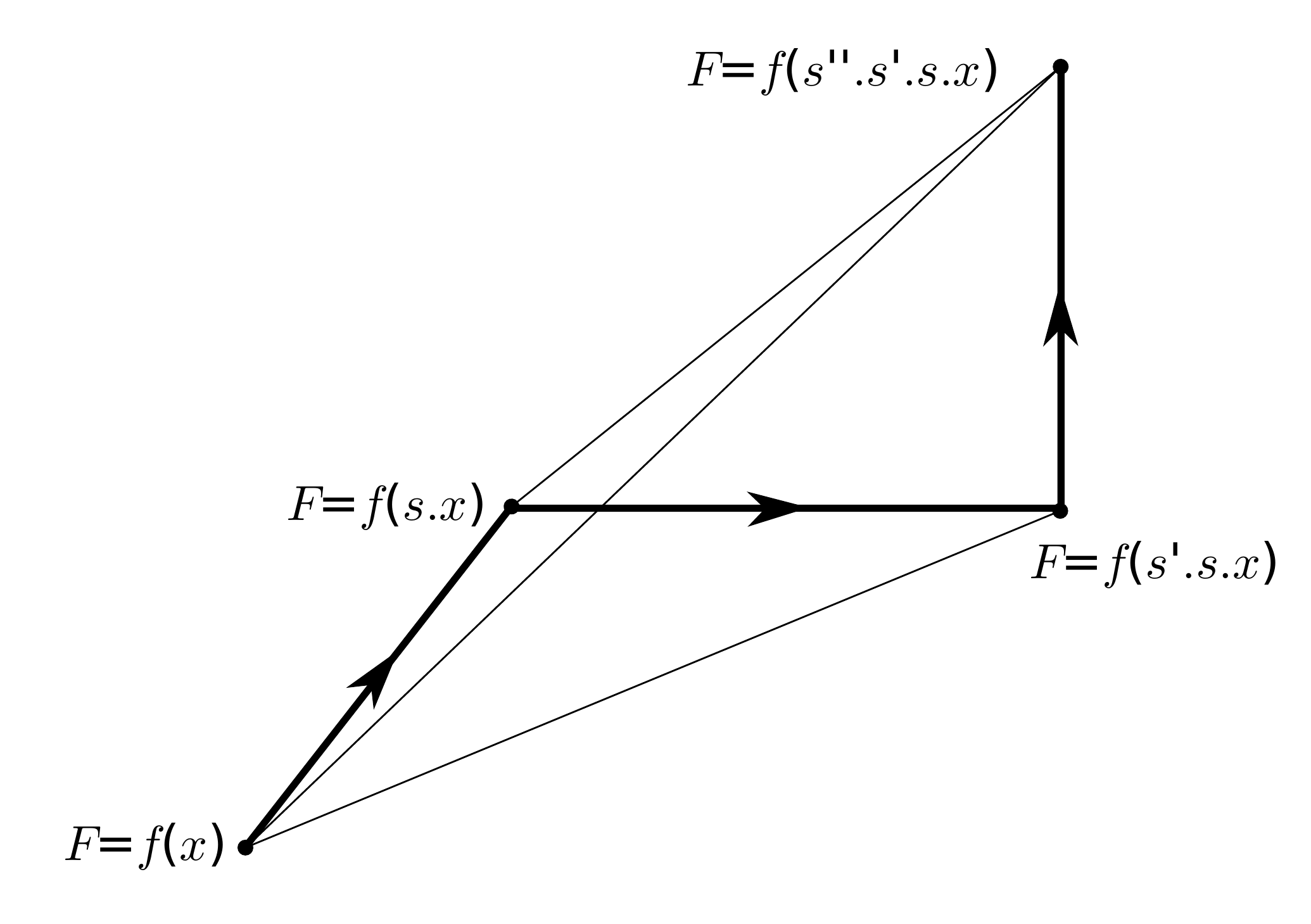}
	\caption{A broken flow line of type $(3;1,1,1)$}
	\label{bild1}
\end{figure}

Fortunately, these broken flow lines do not appear as boundary points of the relevant moduli spaces. To show this we need some combinatorics of simplices:

Recall the construction of the simplex $\Delta^k$ in Subsection \ref{kjfl} and let the vertices of $\Delta^k$ be numbered in the following way: Start with $0=(0, \ldots,0) \in \mathbb{R}^k$ and let $i$ be the vertex which is added to $\Delta^{i-1}$ to obtain $\Delta^i$.

Now given a broken flow line of type $(m,\Gamma)$ from $A^0_x$ to $A^k_y$ we identify it with a sequence of $m+1$ natural numbers starting with $0$, ending with $k$. Every number in this sequence corresponds to the vertex at the critical submanifold the broken flow line passes. Therefore the sequence is strictly increasing. For $m \in \{1,\ldots,k\}$ let $(i_0\negthickspace=\negthickspace0,\ i_1,\ \ldots\ ,i_{m-1},\ i_m\negthickspace=\negthickspace k)$ denote this sequence. 

Then $\Gamma_j=i_{j}-i_{j-1}$ and we have 
\begin{align*}
 \sum_{j=1}^m\Gamma_j&= i_1 - i_0 + i_{2}-i_{1} + \ldots + i_m-i_{m-1} \\
& = i_1 + i_2 - i_1 + i_3 - i_2 + \ldots + i_{m-1} - i_{m-2} + k - i_{m-1}\\
& = k.  
\end{align*}

More generally, for $m=m_b + m_0$:
\begin{equation} \label{comb}
\sum_{j=1}^m\Gamma_j=k.
\end{equation}

Now we are able to prove the final claim: Observe that Corollary \ref{dims} implies
\begin{equation} \label{cc}
\widetilde{\M}_l(z,z')\neq \emptyset \Longleftrightarrow \mu(z')\geq \mu(z)+1-2l.
\end{equation} 
The structure of the broken flow line \eqref{bfl} and the inequality \eqref{cc} imply
\begin{gather*}
 \mu(x)+n_0  = \mu(z_{l_1}), \ \mu(z_{l_1}) \leq \mu(z_{l_1}') + 2l_1-1; \\
\mu(z_{l_1}')+n_1 = \mu(z_{l_2}),\ \mu(z_{l_2}) \leq \mu(z_{l_2}') + 2l_2-1; \\
\ldots \\
\mu(z_{l_{m_b}}) \leq \mu(z_{l_{m_b}'}) + 2l_{m_b}-1 \text{ and} \\
\mu(z_{l_{m_b}'}) + n_{m_b} =  \mu(y).
\end{gather*}
Combining these equations,
\begin{gather*}
\mu(x)+n_0 \leq  \mu(z_{l_1}') + 2l_1-1=  \mu(z_{l_2}) + 2l_1-1 - n_1 \\
\ldots \\
\leq  \mu(z_{l_{m_b}}') + 2l_1-1 + \ldots +2l_{m_b}-1 - n_1 - \ldots - n_{m_b-1}\\
=  \mu(y) + 2l_1-1 + \ldots 2l_{m_b}-1 - n_1 - \ldots - n_{m_b}.
\end{gather*}
Thus, \begin{equation*} \mu(x) \leq  \mu(y) + 2\sum_{i=1}^{m_b}l_i -m_b - \sum_{i=0}^{m_b}n_i.\end{equation*}

Now if $\dim \widetilde{\M}_k(x,y)=0$, that is $\mu(y)=\mu(x)+1-2k$, then the last inequality reads
\begin{align*}
0 \leq\ & 1-2k+2\sum_{i=1}^{m_b}l_i - m_b - \sum_{i=0}^{m_b}n_i \\
\overset{\eqref{comb}}{\Longleftrightarrow} 0 \leq\ & 1-2k + 2k -m_b - \sum_{i=0}^{m_b}n_i\\
\Longleftrightarrow 1 \geq\ & m_b +\sum_{i=0}^{m_b}n_i.
\end{align*}

If $m_b=0$, then $n_0$ is either $0$ or $1$. But $n_0=0$ means there is no flow line at all and $n_0=1$ leads to a contradiction, because we have shown that if $O(w) \in \M_0(x,z')$, i.e. $z'\neq y$, then $m_b$ must be greater zero.
Therefore we conclude that $m_b=1, n_0=n_1=0$, hence $\widetilde{\M}_k(x,y)$ is already compact.

For $\dim \widetilde{\M}_k(x,y)=1$ we have
\begin{align*}
\mu(y)= & \mu(x)+2-2k, \\
\Longrightarrow 0 \leq & 2 - m_b -\sum_{i=0}^{m_b}n_i.
\end{align*}

Again, $m_b=0$ makes no sense. If $m_b=1$, then either $n_0=n_1=0$ or $n_0=0,n_1=1$ and vice versa. In the first case $O(w) \in \M_k(x,y)$ and $\widetilde{\M}_k(x,y)$ is already compact. Moreover, \eqref{comb} shows that the latter case and $m_b=2$ $(n_0=n_1=n_2=0)$ are precisely the $m=2$ broken flow lines stated in the theorem. This completes the proof.  

\end{proof}

\newpage

\subsection{Gluing}
\label{gluing}

We now come to the complementary concept of convergence to broken flow lines: The gluing map. Given $x,y \in \C(f)$ with $\mu(y)-\mu(x)=2-2k$ and a broken flow line of type $(2,\Gamma)$ from $x$ to $y$, the gluing map produces a flow line in the one-dimensional moduli space $\widetilde{\M}_k(x,y)$. Observe that the broken flow lines used to compactify $\widetilde{\M}_k(x,y)$ are ``connected'', that is $u_i^+=u_{i+1}^-$. The gluing map will be defined only for such special broken flow lines - this reduces the gluing to more or less the Morse case and allows us to prove the following theorem with the methods from Weber \cite{We}, i.e. local constructions around the ``connecting point'' of two flow lines: 

\begin{theorem} 
Given critical points $x,y,z_j \in \C(f)$ with $\mu(y)-\mu(x)=2-2k$ and $\mu(z_j)=\mu(x)+1-2j$ and $j\in J=\{0,1,\ldots,k\}$, there exists a real number $\rho_0 > 0$ and for every $j\in J$ an embedding
\begin{gather*}
\Diamond^j: K \times [\rho_0,\infty) \to \widetilde{\M}_k(x,y), \\ (u,v,\rho) \mapsto u\Diamond^j_{\rho}v.
\end{gather*}

Here $K\subset \widetilde{\M}_j(x,z_j) \times \widetilde{\M}_l(z_j,y)$ is the subset of connected flow lines and $l:=k-j$.
\\

The map $\Diamond^j$ satisfies
\begin{equation*}
u\Diamond^j_{\rho}v \longrightarrow (u,v)\ \text{ as }\ \rho \longrightarrow \infty,
\end{equation*} 
and no other sequence in $\widetilde{\M}_k(x,y)\setminus (u\Diamond^j_{[\rho_0,\infty)}v)$ converges to $(u,v)$.

\end{theorem}

\begin{proof}

Fix $j$, and for $u\in \widetilde{\M}_j(x,z_j), v\in \widetilde{\M}_l(z_j,y)$ let $a:=O(v)^-=O(u)^+$ denote the starting point of $v$ in $A^j_{z_j}$. Here $A^j_{z_j}$ is the critical submanifold associated to $z_j$, containing $a$ and lying in $M\times T^k \times \Delta^k$ at the $j$-th vertex. We will work locally around $a$, so let $a=0 \in \mathbb{R}^{n+2k}$ and $A\subset A^j_{z_j}$ be a small neighbourhood of $a$. Furthermore, since $A^j_{z_j}$ is a submanifold of $W$ we can assume without loss of generality that $A=\{p\in \mathbb{R}^{n+2k} | p_{k+1}=\ldots=p_{n+2k}=0\}$. The proof has three steps:
\\

\textbf{1. Local model}
\\

We can locally ``straighten out'' the stable and unstable manifolds of $A$:
\\

Let $E^i=T_aW^i(A)$; the Morse-Bott Lemma implies that we can describe a small tubular neighbourhood $N$ of $A$ by coordinates $(p^0,p^u,p^s)$ with $p^0\in A$, $p^i\in B^i\subset E^i$, $B^i$ small neighbourhoods of zero in $E^i$. Moreover, locally the stable and unstable manifolds of $A$ are graphs - this is a consequence of the ``Stable Manifold Theorem'' (cf. Cresson and Wiggins \cite{CW}). This means, there are smooth maps $\eta_i: A \times B^i \to E^{i*}$ ($i*=u$ if $i=s$ and vice versa) with $\eta_i(p^0,0)=0,\ D\eta_i(p^0,0)=0$ and
\begin{gather*}
W^u(A)\cap N = \{(p^0,p^u) | p^s=\eta_u(p^0,p^u))\},\\
W^s(A)\cap N = \{(p^0,p^s) | p^u=\eta_s(p^0,p^s))\}.
\end{gather*}
The graphs of $\eta_i$ are called \textit{local (un-)stable manifolds} $W^i_{\text{loc}}(A)$. 

Define 
\begin{gather*}
\eta : N \to N, \\ 
\eta(p_0,p_u,p_s):= (p_0,p_u - \eta_s(p_0,p_s),p_s - \eta_u(p_0,p_u)).
\end{gather*} 
The map $\eta$ satisfies $\eta(0)=0$ and $D\eta(0)=id$, hence it is a diffeomorphism on some neighbourhood of zero, ``flattening out'' the local (un-)stable manifolds.

Now define a flow $\Psi_t := \eta \circ \Phi_t \circ \eta^{-1}$. It satisfies $\Psi_t(0)=0,\ D\Psi_t(0)=D\Phi_t(0)$ and a small neighbourhood of zero in $W^i(A,\frac{d\Psi_t}{dt}\restriction_{t=0})$ is a small neighbourhood of zero in $E^i$.
This is our locally flat model.

\begin{figure}[hbt]
	\centering
\includegraphics[width=1.00\textwidth]{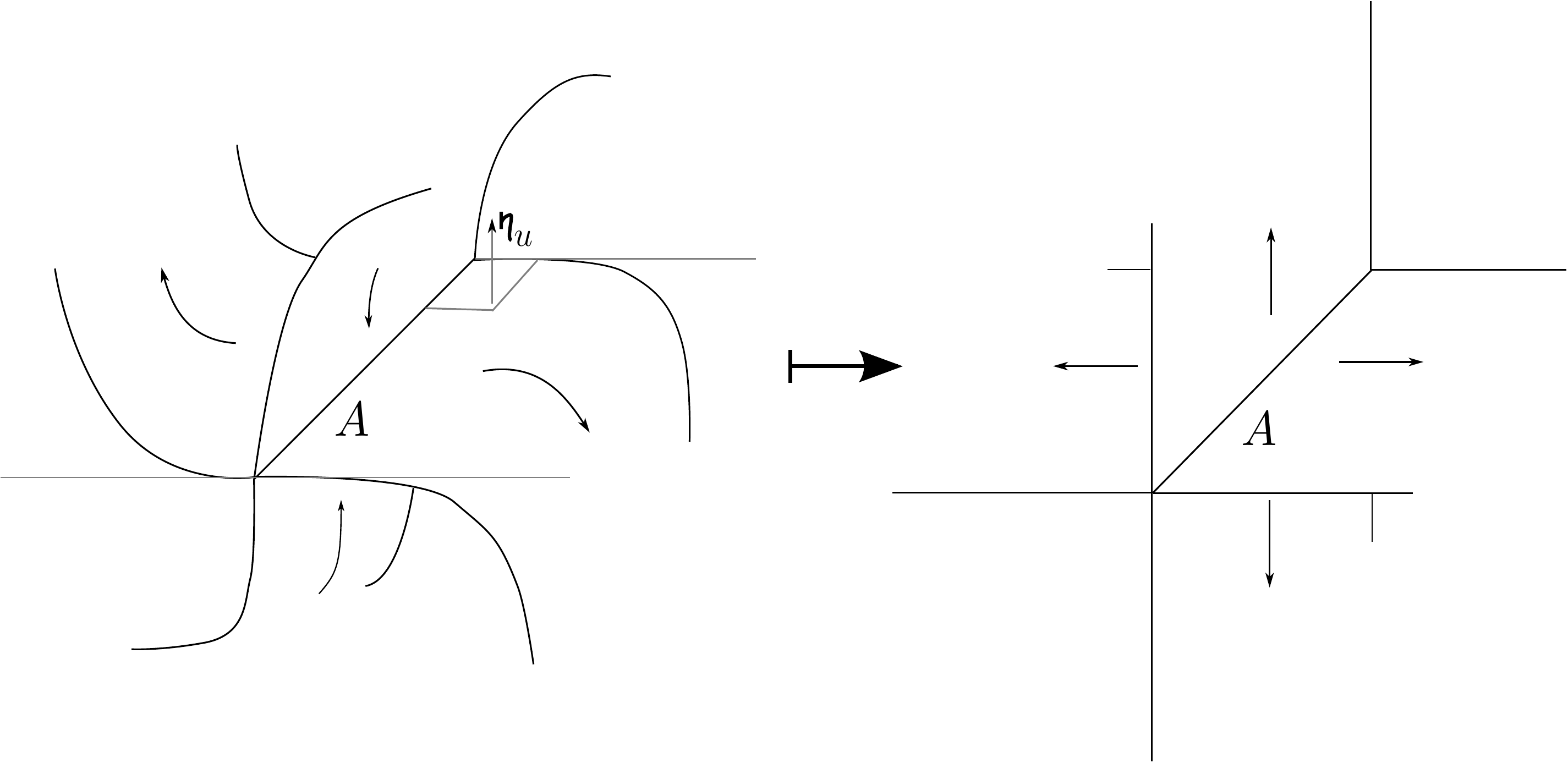}
	\caption{Flattenend out (un-)stable manifolds}
	\label{bild4}
\end{figure}

\textbf{2. Unique intersection point}
\\

Fix closed balls $B^i \subset W^i_{\text{loc}}(A) \subset E^i$ around zero and set $U:=A \times B^u \times B^s $.
For $u\in \widetilde{\M}_j(x,z_j)$ assume that $u \in B^s$ (if not let it flow into $B^s$ with $\Phi_t$).
Choose a $(\mu(z_j)+n+k-j)$-disk $D^{\mu(z_j)+n+k-j}\subset W^u(A^0_x)$ transversally intersecting $A\times B^s$ at $u$. For $t > 0$ denote by $D_t$ the connected component of $\Phi_t(D^{\mu(z_j)+n+k-j})\cap U$ containing $\Phi_t(u)$.

For $v \in \widetilde{\M}_l(z_j,y)$ do the same with a $(\mu(z_j)+j)$-disk in $W^s(A^k_y)$ with respect to the backward flow to obtain $D_{-t}$, the connected component of $\Phi_{-t}(D^{\mu(z_j)+j})\cap U$ containing $\Phi_{-t}(v)$ ($t>0$).
\\
 
Then there exists $t_0>0$, such that for all $t > t_0$ there is a unique intersection point 
\begin{equation*}
p_t:= D_t \cap D_{-t} \cap \{0\}\times B^u\times B^s.
\end{equation*}
\begin{figure}[hbt]
	\centering
\includegraphics[width=0.80\textwidth]{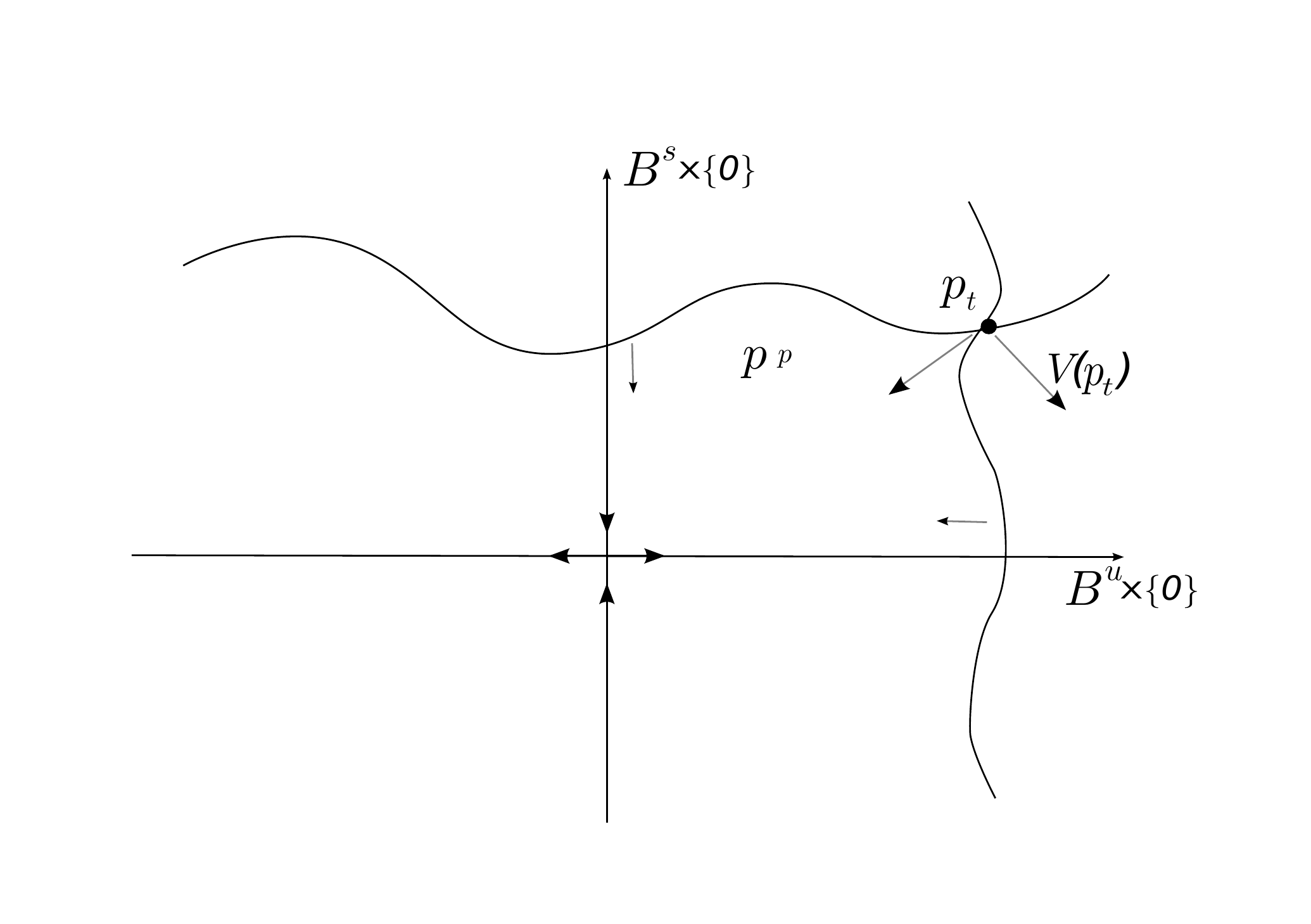}
	\caption{Unique intersection point}
	\label{bild5}
\end{figure}

To prove that $p_t$ is a unique point we want to express both disks as graphs of smooth maps $a_t:A\times B^u \to B^s$ and $b_t:A\times B^s \to B^u$, for then the intersection of both disks in $B^u\times B^s$ equals the fixed point set of the map $c_t: B^u \to B^u$, $q\mapsto b_t(0,a_t(0,q))$. So if $c_t$ is a strict contraction, we are done.
To express both disks as graphs we use the $\lambda$-Lemma (see Cresson and Wiggins \cite{CW}), which in our case can be stated in the following way:

Since $D^{\mu(z_j)+n+k-j}$ intersects $A\times B^s$ transversally, there exists for every $\epsilon > 0$ a $t_0 > 0$, such that $D_t$ is $C^1\ \epsilon$-close to $A\times B^u$ for all $t \geq t_0$. A simliar statement holds for $D_{-t}$.
This is equivalent to the existence of diffeomorphisms
\begin{gather*}
\alpha_t: A\times B^u \to D_t,\quad p=(p^0,p^u) \mapsto (\alpha_t^0(p),\alpha_t^u(p),\alpha_t^s(p)), \\
\beta_t: A\times B^s \to D_{-t},\quad p'=(p'^0,p'^s) \mapsto (\beta_t^0(p'),\beta_t^u(p'),\beta_t^s(p')),
\end{gather*}
satisfying for all $p \in A\times B^u$ and $p' \in A \times B^s $
\begin{gather*}
\|(p,0)-(\alpha_t^0(p),\alpha_t^u(p),\alpha_t^s(p))\| < \epsilon,\ ||(id,0)-(D\alpha_t^0(p),D\alpha_t^u(p),D\alpha_t^s(p)) || < \epsilon, \\
\|(p',0)-(\beta_t^0(p'),\beta_t^s(p'),\beta_t^u(p'))\| < \epsilon, ||(id,0)-(D\beta_t^0(p'),D\beta_t^s(p'),D\beta_t^u(p')) || < \epsilon.
\end{gather*}

Hence for $\epsilon$ small, the maps $p\mapsto (\alpha_t^0(p),\alpha_t^u(p))$ and $p' \mapsto (\beta_t^0(p'),\beta_t^s(p'))$ are invertible and the graph maps are given by 
\begin{equation*}
a_t:=\alpha_t^s \circ (\alpha_t^0,\alpha_t^u)^{-1},\quad b_t:=\beta_t^u \circ (\beta_t^0,\beta_t^s)^{-1}.
\end{equation*}
Now $c_t$ is a strict contraction, because
\begin{align*}
|| Dc_t(q) ||&= || D\big( b_t\circ(0\times id)\circ a_t \circ (0\times id) \big)(q) || \\
&= || Db_t\big(0,a_t(0,q)\big) \cdot Da_t(0,q)  || \\
& =|| D(\beta_t^u \circ (\beta_t^0,\beta_t^s)^{-1})\big(0,a_t(0,q)\big) \cdot D(\alpha_t^s \circ (\alpha_t^0,\alpha_t^u)^{-1})(0,q) || \\
&= ||D\beta_t^u\big(\big(\beta_t^0,\beta_t^s\big)^{-1}\big(0,a_t(q)\big)\big) \cdot D(\beta_t^0,\beta_t^s)^{-1}\big(0,a_t(q)\big)\\
& \quad  \ \cdot D\alpha_t^s\big((\alpha_t^0,\alpha_t^u)^{-1}(0,q)\big) \cdot D(\alpha_t^0,\alpha_t^u)^{-1}(0,q)|| \\
&\leq ||D\beta_t^u\big(\big(\beta_t^0,\beta_t^s\big)^{-1}\big(0,a_t(q)\big)\big)||\cdot ||D(\beta_t^0,\beta_t^s)^{-1}\big(0,a_t(q)\big)|| \\
&\quad \cdot ||D\alpha_t^s\big((\alpha_t^0,\alpha_t^u)^{-1}(0,q)\big)|| \cdot||D(\alpha_t^0,\alpha_t^u)^{-1}(0,q)|| \\
&\leq \epsilon^2 \cdot||D(\beta_t^0,\beta_t^s)^{-1}\big(0,a_t(q)\big)|| \cdot ||D(\alpha_t^0,\alpha_t^u)^{-1}(0,q)||\ \text{ ($\epsilon$-closeness)} \\
&\leq \epsilon^2\frac{1}{(1-\epsilon)^2} \ <1
\end{align*} 
holds for all $\epsilon < \frac{1}{2}$. The last inequality is due to a corollary of the properties of a Neumann series: 
\begin{equation*}
||id - T|| < 1 \Longrightarrow T \text{ is invertible and } ||T^{-1}|| < \frac{1}{1-q}, \text{ where } q=||id-T||.
\end{equation*}

The Banach Fixed Point Theorem asserts that $ c_t$ possesses a unique fixed point $p_t$ with $\|p_t\| < \sqrt{2}\epsilon$ for all $t > t_0$.
\\

\textbf{3. The gluing map}
\\

Set $\rho_0:= t_0$ and
\begin{equation*}
u\Diamond^j_{\rho}v:=p_{\rho}\quad \forall \rho \in [\rho_0,\infty). 
\end{equation*}

Clearly $p_{\rho} \in \widetilde{\M}_k(x,y)$ for all $\rho > \rho_0$. It remains to show that $\Diamond^j$ is an embedding and the convergence property asserted in the theorem.

The vector field $V$ is transverse to $D_t$ and $D_{-t}$ (otherwise choose $D^{\mu(z_j)+n+k-j}$ and $D^{\mu(z_j)+j}$ smaller in the previous steps). This implies $\frac{dp_t}{dt} \neq 0$, because $D_t$ and $D_{-t}$ travel in different (time-)direction, hence are displaced from themselves by the flow. In other words, the map $ t \mapsto p_t$ is an immersion onto $\M_k(x,y)$. 

Moreover, because of the different travelling directions, $\frac{dp_t}{dt}$ and $V(p_t)$ cannot be linearly dependent (cf. Fig. \ref{bild5}) - so $p_t$ varies with $t$ along $\widetilde{\M}_k(x,y)$ and $\dim{\widetilde{\M}_k(x,y)}=1$ ensures that there is no possibility of self-intersections or self-returns preventing $ t \mapsto p_t$ from being a homeomorphism onto its image. Therefore 
\begin{gather*}
\Diamond^j: \widetilde{\M}_j(x,z_j) \times [\rho_0,\infty) \times \widetilde{\M}_l(z_j,y) \to \widetilde{\M}_k(x,y), \\
(u,v,\rho) \mapsto u\Diamond^j_{\rho}v
\end{gather*}
is an embedding.

To prove the convergence statement choose a sequence 
\begin{equation*}
\epsilon_l \longrightarrow 0,\ \epsilon_l > 0\quad  \forall l \in \mathbb{N}.
\end{equation*} 
The $\lambda$-Lemma produces a sequence $(t_{0,l})_l$, such that $D_t$ is $C^1 \ \epsilon_l $-close to $A\times B^u$ for $t > t_{0,l}$ and the same holds for $D_{-t}$ and $A\times B^s$.

For any sequence of sufficiently large real numbers $t_i \longrightarrow \infty$ choose a subsequence $(t_{0,l_i})_i$ of $(t_{0,l})_l$, such that $t_i \geq t_{0,l_i}$. Then 
\begin{gather*}
D_{t_i} \text{ is }  C^1\ \epsilon_{l_i} \text{-close to } A\times B^u, \\
\Longrightarrow \|p_{t_i}\| < \sqrt{2}\epsilon_{l_i} \text{, converging to zero as } i \longrightarrow \infty. \\
\Longrightarrow \|p_t\| \longrightarrow 0 \text{ as } t \longrightarrow \infty, 
\end{gather*}
and similiar for $D_{-t}$.
Using the same arguments as in the proof of the Compactness Theorem we conclude that
\begin{equation*}
u\Diamond^j_{\rho}v \longrightarrow (u,v) \text{ as } \rho \longrightarrow \infty 
\end{equation*}
and no other sequence converges to $(u,v)$ since the intersection point $p_t$ was shown to be unique.

Finally, repeat this construction to obtain $\Diamond^j$ for all $j \in J$ and set 
\begin{equation*}
\rho_0:=\max_{j\in J}(\rho_0(j)).
\end{equation*}

\end{proof}

The essence of the last two sections is the following corollary:

\begin{cor} \label{bound}
For $\dim \widetilde{\M}_k(x,y)=1$ the connected components of the compactified moduli spaces are diffeomorphic either to $S^1$ or $[0,1]$ and in the second case every boundary component corresponds to precisely one broken orbit $(u,v)$ of type $(2,\Gamma)$ with $\Gamma=(i,j) \ (i+j=k)$.
\end{cor}

\begin{proof}
Assume $\widetilde{\M}_k(x,y)$ is connected; since it is a one-dimensional manifold without boundary it must be diffeomorphic to $S^1$ or $(0,1)$. By the Compactness Theorem it is compactifiable using broken flow lines and the last statement in the Gluing Theorem shows that there is exactly one broken flow line corresponding to a boundary point of $[0,1]$. 
\end{proof}

\newpage

\section{$S^1$-Equivariant Morse Cohomology}
\label{emcoh} 

So far we have completed and justified the definition of the equivariant Morse chain groups and the differential $d_{S^1}$ associated to a Morse-Smale pair $(f,g)$ on a closed manifold $M$. In this chapter we show $d_{S^1}^2=0$ and that the homology of this complex equals the equivariant cohomology of $M$. 

\subsection{The $S^1$-equivariant Morse complex}
\label{temc}

\begin{theorem} \label{ds12}

Recall the definition of $d_{S^1}$ in $\eqref{deq}$. For a general equivariant Morse chain $\alpha(T)=\sum_i a_i \cdot T^i$ with $a_i \in CM^*$ we have
\begin{equation*}
d_{S^1}^2 \alpha(T) =0.
\end{equation*}

\end{theorem}

\begin{proof}
It suffices to show this for a generator $x\in \C(f)$ of $CM^*$ - the general case then follows from linearity and $T$-independence of $d_{S^1}$. 
\begin{align*}
d_{S^1}(d_{S^1}x)= & d_{S^1}( dx + \sum_k R_{2k-1}x \cdot T^k) \\
= & d(dx) + \sum_k d(R_{2k-1}x)\cdot T^k + \sum_l R_{2l-1}(dx) \cdot T^l \\
&+ \sum_l R_{2l-1}\big(\sum_k R_{2k-1}x\cdot T^k\big)\cdot T^l \\
= &0 + \sum_k d(R_{2k-1}x)\cdot T^k + \sum_l R_{2l-1}(dx)\cdot T^l \\
& + \sum_{i,j} R_{2i-1}(R_{2j-1}x)\cdot T^{i+j} \\
= & \sum_k \big(d(R_{2k-1}x) + R_{2k-1}(dx)\big)\cdot T^k \\
& + \sum_{i,j} R_{2i-1}(R_{2j-1}x)\cdot T^{i+j}.
\end{align*}

This expression vanishes iff for all $k$:
\begin{equation*} 
(d\circ R_{2k-1})(x) + (R_{2k-1}\circ d)(x)+ \sum_{i+j=k} (R_{2i-1}\circ R_{2j-1})(x)=0.
\end{equation*}

Insert definitions \eqref{rk} and $\eqref{md}$ of $R_{2k-1}$ and $d$: 

\begin{align*}
d(R_{2k-1}x) &= d(\sum_z n_k(x,z)z) \\
&= \sum_y \sum_z n_k(x,z) n_0(z,y)y \\
&=  \sum_y \sum_z \big( | \{u \in \widetilde{\M}_k(x,z)\}|\cdot| \{v \in \widetilde{\M}_0(z,y)\} | \bmod{2}\big) y  \\
&= \sum_y \sum_z \big( | \{(u,v) \in \widetilde{\M}_k(x,z) \times \widetilde{\M}_0(z,y)\} | \bmod{2}\big) y, \\
R_{2k-1}(dx) &= R_{2k-1} ( \sum_{z'}n_0(x,z') z' )\\
&= \sum_y \sum_{z'} \big( | \{(u,v) \in \widetilde{\M}_0(x,z') \times \widetilde{\M}_k(z',y)\} | \bmod{2}\big)y, \\
R_{2i-1}(R_{2j-1}x) &= R_{2i-1}(\sum_{z_j} n_j(x,z_j)z_j) \\
&=\sum_y  \sum_{z_j} n_j(x,z_j)n_i(z_j,y) y \\
&=\sum_y \sum_{z_j} \big( | \{(u,v) \in \widetilde{\M}_j(x,z_j) \times \widetilde{\M}_i(z_j,y)\} |\bmod{2}\big)y,
\end{align*}

with $\mu(z)=\mu(x)-2k+1$, $\mu(z')=\mu(x)+1$, $\mu(z_j)=\mu(x)-2j+1$ and $\mu(y)=\mu(x)-2k+2$. 

Fix $y$; by Corollary \ref{bound} every summand corresponds to a boundary component of the compactification of the one-dimensional moduli space $\widetilde{\M}_k(x,y)$. Hence,

\begin{gather*}
(d\circ R_{2k-1})(x) + (R_{2k-1}\circ d)(x)+ \sum_{i+j=k} (R_{2i-1}\circ R_{2j-1})(x)= \\
\sum_y \big( \sum_{(u,v) \in \partial \widetilde{\M}_k(x,y)} 1\bmod{2}  \big)y=0
\end{gather*}
vanishes, because boundary components (i.e. broken flow lines) always come in pairs and therefore their sum is zero modulo $2$.

\end{proof}

\begin{remark} \label{orient}
For $S^1$-equivariant Morse cohomology with $\mathbb{Z}$-coefficients one would need an orientation of the moduli spaces. Then, mimicing the definition of the differential in ordinary Morse homology,
\begin{equation*}
n_i(x,y):=\sum_{u \in \widetilde{\M}_k(x,y)}\epsilon(u)
\end{equation*}
where $\epsilon(u)=\pm 1$ wether or not $u$ is positively oriented with relation to the orientation on $\widetilde{\M}_k(x,y)$. Then, two broken flow lines corresponding to two boundary components of a connected component of $\widetilde{\M}_k(x,y)$ come with alternating signs, such that their sum vanishes. 

The problem here is to assign an orientation to the moduli spaces: In the ordinary Morse case the (un-)stable manifolds are contractible and therefore orientable (inducing an orientation of the corresponding moduli spaces), whereas in the equivariant case they are $k$-tori-families of open discs which possibly includes non-orientable objects such as the Moebius strip or the Klein bottle. 
\end{remark}

\subsection{$H_*(CM_{S^1}^*,d_{S^1})$}
\label{eqmoco}

Since $d_{S^1}^2=0$ we can take homology of the complex $(CM_{S^1}^*,d_{S^1})$:

\begin{definition}
The \textit{$S^1$-equivariant Morse cohomology groups} (with $\mathbb{Z}_2$-coefficients) are defined as

\begin{equation*}
HM_{S^1}^n:= \ker ( d_{S^1}\restriction_{CM_{S^1}^n} )/  \text{ im}(d_{S^1}\restriction_{CM_{S^1}^{n-1}}).
\end{equation*}

\end{definition}

\begin{theorem} \label{emt}
$HM^*_{S^1}(M)=H^*_{S^1}(M)$. 
\end{theorem}

\begin{proof}
Since homology commutes with direct limits (see Spanier \cite{Sp}) we understand the right hand side of this equation as direct limit of the directed system ($H^*(X^i),p_{ij})$, where $X^m$ is the total space of the fiber bundle
\begin{equation*}
\begin{CD}
 (S^{2m+1}\times M)/S^1 \\
 @VV \pi_m V \\
 \mathbb{CP}^m,
\end{CD}
\end{equation*}
and the maps $p_{ij}: H^*(X^i) \to H^*(X^j)$ are induced by the projections $S^{2j+1}\to S^{2i+1}$. On the other hand, $ M_{S^1} = \varinjlim X^m$ is the direct limit of the directed system $(X^i,\iota_{ij})$, where $\iota_{ij}: X^i \to X^j$ are the inclusions induced by $S^1 \subset S^3 \subset \ldots \subset S^{\infty}$. Therefore,
\begin{gather*}
H^*_{S^1}(M)=H^*(\varinjlim X^m)= \varinjlim H^*(X^m).
\end{gather*}
Let $\mathcal{C}[T]$ denote the equivariant Morse cochains. We need to show that the restricted complex 
\begin{equation*}
(\mathcal{C} [T] / (T^{m+1}),d+\sum_{k=1}^{m}R_{2k-1}\cdot T^k)
\end{equation*}
computes the cohomology of $X^m$. We will do this in the case $m=1$, but first we introduce the idea in general:
\\ 

Let $\phi_m$ be the standard Morse function on $\mathbb{CP}^m$,
\begin{gather*}
\phi_m: \mathbb{CP}^m \to \mathbb{R}, \\
[z_0: \ldots: z_m] \mapsto \sum_{k=1}^m k\|z_k\|^2.
\end{gather*}
$\phi_m$ has critical points $\displaystyle{z^0=[1:0:\ldots],\ \ldots\ , z^m=[0:\ldots:1]}$ with Morse indices $\mu(z^i)=2i$. Therefore the Morse differential vanishes in all degrees and every moduli space $\widetilde{\M}(z^i,z^{i+1})$ is isomorphic to $S^1$ - this is a consequence of the cell structure of $\mathbb{CP}^m$:
\begin{equation*}
\mathbb{CP}^m=e^0 \cup e^2 \cup \ldots \cup e^{2m}. 
\end{equation*}
$\phi_m$ lifts to a Morse-Bott function $\Phi_m:=\phi_m\circ \pi_m$ on $X^m$ with critical set 
\begin{equation*}
\C(\Phi_m)=\bigcup_{i=0}^m \pi_m^{-1}(z^i) := \bigcup_{i=0}^m M_i,\ \mu(M_i)=2i.
\end{equation*} 

We use FLWC to obtain $H^*(X^m)$: Choose a Morse-Smale pair $(f,g)$ on $M$ and a connection on the principal bundle 
\begin{equation*}
\begin{CD}
 S^{2m+1} \\
 @VV p_m V \\
 \mathbb{CP}^m.
\end{CD}
\end{equation*}
and let $A_m$ denote the induced connection on the associated bundle $\pi_m: X^m \to \mathbb{CP}^m$. Now identify $\pi_m^{-1}(z^0)$ with $M$ and use parallel transport along the unique horizontal lift with respect to $A_m$ of a flow line $u$ from $z^0$ to $z^1$ in $\mathbb{CP}^m$ to identify the two fibers over $z^0$ and $z^1$. Continue with a flow line from $z^1$ to $z^2$ and so on until all critical fibers $M_i$ are identified with $M$. This defines a Morse-function on $\C(\Phi_m) \cong \dot{\cup}_m M$, which is just $f$ on every fiber. 
\\

We need two important properties of the bundles $p_m$: 
\\

\textbf{1.} Using (the lift of) another flow line $u'$ from $z^0$ to $z^i$ the difference between the corresponding parallel transports of $p \in M$ is given by the $S^1$-action on $M$: If $q\in M$ is the parallel transport $\mathcal{P}^{A_m}_u(p)$ of $p$ along $u$, then $\mathcal{P}^{A_m}_{u'}(p)=s.q $ for some $s\in S^1$. Observe that the same holds in the associated bundle $\pi_m$, because $A_m$ induces parallel transport in every associated bundle by
\begin{equation*}
{\mathcal{P}_{\text{ass}}^{A_m}}_u( [p,q] ) := [\mathcal{P}^{A_m}_u(p),q].
\end{equation*}

\textbf{2.} Since for $l\leq m-i$ all moduli spaces $\widetilde{\M}(z^i,z^{i+l})$ are isomorphic and flow lines of $\Phi_m$ are lifts of flow lines of $\phi_m$, the following holds: If there is a flow line of $\Phi_m$ from $M_i$ to $M_{i+l}$, then there is also one from $M_j$ to $M_{j+l}$ for all $i,j \in \{0,\ldots, m\},\ i+l$ and $j+l \leq m$. In other words, when considering flow lines with $l$ cascades from $M_i$ to $M_{i+l}$, it suffices to study those from $M_0$ to $M_l$.
\\

Now let $x\in \C(f)$ and recall the grading of Section \ref{flwc}, 
\begin{equation*}
\lambda(x)=\mu_{\Phi_m}(x)+\mu_f(x)=\mu(x)+2i \ \text{ if $x \in M_i$},
\end{equation*} 
which we are able to express in this special case with an independent variable $T$ of degree $2$, such that the complex $(CC^*,d^c)$ takes the following form:
\begin{gather*}
CC^*:= CM^*\otimes \mathbb{Z}_2[T] \cong CM_{S^1}^*, \\
d^c (x\cdot T^i) :=\sum_k d_k x \cdot T^i, \\
d_k x := \sum_{\mu(y)=\mu(x)+1-2k}n_k(x,y)y\cdot T^k,
\end{gather*}
where $n_k(x,y)$ is the algebraic count of flow lines with $k$ cascades from $x\in M_0$ to $y\in M_k$. 

A flow line with $k$ cascades on $X^m$ translates in this special case into a flow line $u$ of $\nabla_g f$ on $M$, such that there is $\underline{T}:=(t_1, \ldots, t_{2k})\in \mathbb{R}^{2k}$ with $-\infty =:t_0 < t_1< \ldots < t_{2k} <  t_{2k+1}:=\infty $ and for even $i$:
\begin{gather}
\dot{u}(t)= \nabla_g f\big(u(t)\big) \text{ for $t\in (t_i,t_{i+1})$ } \text{and } u(t_i)=s_i.u(t_{i-1}) \text{ with $s_i \in S^1.$ } \label{nsflwc}
\end{gather}   

Here the $s_i$ play the role of the cascades due to the first property of $p_m$ mentioned above. 

This is the ``non-smooth'' picture, where a $k$-jump flow line consists of $k+1$ solutions of $\dot{u}=\nabla_g f(u)$ matched together by $S^1$-orbits of their starting and ending points. In other words, there is a Morse-like complex computing $H^*(X^m)$ which is generated by the critical points of $f$ (with a different grading) and a differential counting piecewise smooth curves consisting of flow lines of $\nabla_gf$ and orbits of the $S^1$-action on $M$.

It remains to show that this is equivalent to the ``smooth'' picture we have used in the construction of the $S^1$-equivariant Morse complex in Chapters $3$ and $4$, i.e. using homotopies between $f$ and $f\circ \sigma_s$. N.B. Since the homology of the complex associated to FLWC is independent of the chosen Morse-Smale pair $(f,g)$, the same will be true for the equivariant Morse complex.
\\

From now on let $m=1$ (like in Chapter $4$ we omit the subscript $1$ in $W,F,V$):
\\

First we show property 1 of $p_1$, the Hopf-bundle (with the standard connection), by direct computation: 
Let
\begin{equation*}
h: S^3 \to S^2\cong \mathbb{C}^*, \quad (z_1,z_2) \mapsto \frac{z_1}{z_2},
\end{equation*}  
be the Hopf map and let 
\begin{equation*}
z_1=\exp(i\xi_1) \sin \eta, \ z_2=\exp(i\xi_2)\cos \eta \quad \big(\xi_i \in [0,2\pi),\ \eta \in (0,\frac{\pi}{2}]\big).
\end{equation*}
The fibers over the north- and southpole $N,S$ of $S^2$ are given by 
\begin{equation*}
h^{-1}(\infty)=(\exp(i\xi_1),0),\ h^{-1}(0)=(0,\exp(i\xi_2)).
\end{equation*}
Flow lines $u_{\gamma}$ from $S$ to $N$ (i.e. great circles in $S^2$) are given by $u_{\gamma}(t)=\exp(i\gamma)\cdot t$ with $\gamma \in [0,2\pi)$. Their (horizontal) lifts $U_{\gamma}=(z_1(t),z_2(t))$ satisfy:
\begin{equation*}
\eta=\tan^{-1}t \ \text{ and } \ \xi_1 = \begin{cases} \gamma + \xi_2 & \text{ if $\xi_1 \geq \xi_2$,} \\
\gamma + \xi_2 - 2\pi & \text { if $\xi_1 < \xi_2$}. \end{cases}
\end{equation*} 
Thus, 
\begin{equation*}
U_{\gamma}(t)=\big( \exp\big(i\xi_1(\xi_2)\big) \sin(\tan^{-1}t), \exp(i\xi_2)\cos(\tan^{-1}t)   \big),
\end{equation*}
with $U_{\gamma}(0)=(0,\exp(i\xi_2))$ and $U_{\gamma}(\infty)=(\exp\big(i\xi_1(\xi_2)\big),0) $ for all $\gamma \in [0,2\pi)$. 

Comparing $U_0$ with $U_{\gamma}$ we conclude that
\begin{gather*}
U_{\gamma}(0)=U_0(0), \\
U_{\gamma}(\infty)=\exp(i\gamma) \cdot U_0(\infty).
\end{gather*}
The assertion now follows, because parallel transport is $S^1$-equivariant and bijective (with its inverse given by parallel transport along the same curve with reversed time). 
\\

To complete the proof of Theorem \ref{emt} we need two more steps: First we show that for ``small'' homotopies (this is made precise down below), both pictures are equivalent, i.e. there is a one-to-one correspondence between solutions of the smooth and non-smooth systems. Then we show that the smooth picture is independent of the chosen homotopy. Combining these steps proves the theorem.
\\

\textbf{1. The bijection}

A flow line $u$ with a non-trivial jump in the above sense is a solution of
\begin{align*}
\dot{u}&=\nabla_g f(u), \\
u^-=x,\ u^+&=y, \ x,y\in \C(f), \\ 
\exists \ t_0 \in \mathbb{R}, s\in S^1 : &\lim_{t\to t_0, t < t_0} u(t)=s. \lim_{t\to t_0, t > t_0}u(t).
\end{align*}

We can act with $s^{-1}$ on the second part of $u$ to obtain an alternative description ($t_0=0$) of this system:

\begin{equation} \label{nsp}
\begin{split}
\dot{u}\restriction_{(-\infty,0)}&=\nabla_g f(u\restriction_{(-\infty,0)}), \\
\dot{u}\restriction_{(0,\infty)}&= \nabla_g \big( f\circ \sigma_s \big) (u\restriction_{(0,\infty)}), \\
u^-=x,\ u^+&=s^{-1}.y, \ x,y\in \C(f). \end{split}
\end{equation}

On the other hand, in our smooth picture $u$ (more precisely, the projection onto $M$ of $u: \mathbb{R} \to W$ with $\dot{u}=V(u)$) solves:
\begin{equation} \label{sp} 
\begin{split}
\dot{u}(t)&=\nabla_g F\big(u(t),s,t\big), \\
u^-=x,\ u^+&=s^{-1}.y, \ x,y \in \C(f). \end{split}
\end{equation}
\\

Let $F_{\rho}(p,s,t)=\phi_{\rho}(t) f(p) + (1-\phi_{\rho}(t))f(s.p)$ be a one-parameter family of homotopies from $f$ to $f\circ \sigma_s$, where $\phi_{\rho}: \mathbb{R} \to \mathbb{R}$ is a smooth function satisfying 
\begin{equation*}
\phi_{\rho}(t)=\begin{cases}
    1 & \text{if $ t < -\rho $ },\\
    0 & \text{if $ t > \rho$ }.
  \end{cases}
\end{equation*}
and let $V_{\rho}$ be the associated vector field $\nabla_g F_{\rho}$.
For a given $\rho >0$, 
\begin{equation} \label{sp2}
\begin{split}
\dot{u}_{\rho}(t)&=V_{\rho}\big(u_{\rho}(t),s,t\big), \\ 
u_{\rho}^-=x,\ u_{\rho}^+&=s^{-1}.y, \ x,y \in \C(f), \end{split}
\end{equation}
is just equation $\eqref{sp}$ for $F=F_{\rho}$, whereas for $\rho=0$, 
\begin{equation} \label{nsp2}
\begin{split} 
\dot{u}_0(t)&=V_0\big( u_0(t),s,t\big), \\ 
u_0^-=x,\ u_0^+&=s^{-1}.y, \ x,y\in \C(f), \end{split}
\end{equation}
corresponds to the non-smooth picture of equation $\eqref{nsp}$.
\\

\textbf{a.} Let $u_0$ be a solution of $\eqref{nsp2}$: Since the vector field $V_0$ is discontinuous at $t=0$ we ``smoothen it out'' by considering an equivalent version of the above ODE:
Set $W_{\rho}(p,s,t):=a(t)\cdot V_{\rho}\big(p,s,b(t)\big)$ with $b \in C^{\infty}(\mathbb{R},\mathbb{R}), \ b^{(k)}(0)=0$ for all $k\in \mathbb{N}$ and
\begin{equation*}
\dot{b}(t)=a(t)=\begin{cases}
    1-\exp(\frac{-1}{\epsilon^2-t^2}) & \text{ if $ |t| < \epsilon $ },\\    
0 & \text{ else },
  \end{cases}
\end{equation*}
for some small $\epsilon > 0$.
We need to reparametrize solutions $u_{\rho}$ as well:
\begin{equation*}
v_{\rho}(t):=u_{\rho}\big( b(t) \big).
\end{equation*}
Then for all $\rho \in \mathbb{R}^+_0$:
\begin{align*}
\dot{v}_{\rho}(t) & = \dot{b}(t) \dot{u}_{\rho}\big(b(t)\big) \\
& = \dot{b}(t) V_{\rho}\big( u_{\rho}\big(b(t)\big),s,b(t) \big) \\
& = a(t) V_{\rho}\big(v_{\rho}(t),s,b(t) \big) \\
& = W_{\rho}\big( v_{\rho}(t),s,t \big).
\end{align*}

By construction, $W_{\rho}$ is smooth in all variables. From continuous parameter-dependence of smooth ODEs it follows, that for $\rho$ in a neighbourhood of $0$ there exists a unique solution of $\dot{v}_{\rho}(t)=W_{\rho}\big( v_{\rho}(t),s,t \big)$. $C^0$ closeness of $v_{\rho}$ to $v_0$ implies that for $\rho$ small enough $v_{\rho}^-=v_0^-$ and $v_{\rho}^+=v_0^+$, since critical points are isolated. Transforming back with $b^{-1}$ (defining $b^{-1}(0):=0$, because at $0$, where $\dot{b}=0$, $b$ fails to be a diffeomorphism) we have established one direction of the one-to-one correspondence between solutions of $\eqref{sp2}$ and $\eqref{nsp2}$.
\\

\textbf{b.} For the other direction choose $0<\rho < \min_{x,y \in \C(f),s\in S^1}d(s.x,s.y)$ (exists since $\C(f)$ and $S^1$ are compact sets) and let $u_{\rho}$ be a solution of $\eqref{sp2}$. This ensures $\lim_{t\to \pm \infty} u_{\rho'}(t)= \lim_{t\to\pm \infty}u_{\rho}(t)$ for all $0<\rho'< \rho$. This is the reason why we called $F_{\rho}$ a ``small'' homotpy. 

Define $u_0$ on $(-\infty,0)$ and $(0,\infty)$ by
\begin{equation*}
u_0(t)=\lim_{\rho \to 0}u_{\rho}(t). 
\end{equation*} 
Clearly \begin{gather*}
\dot{u}_0(t)\restriction_{(-\infty,0)}=\nabla_g f\big(u_0(t)\restriction_{(-\infty,0)}\big), \\
\dot{u}_0(t)\restriction_{(0,\infty)}=\nabla_g \big(f\circ \sigma_s\big) \big(u_0(t)\restriction_{(0,\infty)}\big). 
\end{gather*}
Continuity of $u_0$ at $t=0$ follows from
\begin{align*}
 d(u_{\rho}(-\rho),u_{\rho}(\rho))&\leq L(u_{\rho}\restriction_{[-\rho,\rho]}) \\
&= \int_{-\rho}^{\rho} \| \dot{u}_{\rho}(t) \| dt \\
&= \int_{-\rho}^{\rho} \| \phi_{\rho}(t)\cdot \nabla_g f \big(u_{\rho}(t)\big) + \big(1-\phi_{\rho}(t)\big) \cdot \nabla_g \big(f\circ \sigma_s\big)\big(u_{\rho}(t)\big)  \|dt \\
&\leq \int_{-\rho}^{\rho} \| \phi_{\rho}(t)\cdot \nabla_g f \big(u_{\rho}(t)\big) \| + \|\big(1-\phi_{\rho}(t)\big) \cdot \nabla_g \big(f\circ \sigma_s\big)\big(u_{\rho}(t)\big)  \|dt \\
&\leq K_1 \cdot \rho + K_2 \cdot \rho \ \to 0 \text{ as $\rho \to 0$}.
\end{align*}

Here the last line follows from smoothness of $f$ and $\phi$ and compactness of $M$. Therefore, $u_0$ is a continuous piecewise smooth solution of $\eqref{nsp}$.
\\

Putting all together we have established a bijection between the set of solutions of $\eqref{nsp}$ and $\eqref{sp}$. Thus, both associated complexes (consisting of the same groups) have the same differential operator if we choose the homotopy accordingly. It follows that for this special homotopy 
\begin{equation*}
H(\mathcal{C} [T] / (T^2),d+R_1\cdot T)\cong H^*(X^1).
\end{equation*}
\\

\textbf{2. Independence of the chosen homotopy}
\\

Let $F^{\alpha}$ and $F^{\beta}$ be two $S^1$-families of homotopies between $f$ and $f\circ \sigma_s$. We construct a cochain map $K$ between the two associated equivariant Morse complexes $C^{\alpha}:=(CM_{S^1}^*,d^{\alpha}_{S^1})$ and $C^{\beta}:=(CM_{S^1}^*,d_{S^1}^{\beta})$ which induces an isomorphism on homology. 
\\

Without loss of generality assume that $F^{\alpha}$ and $F^{\beta}$ satisfy the conditions of Definition \ref{V_1} and let 
\begin{equation*}
\tilde{H}: [0,1] \times W \to \mathbb{R},\quad  \tilde{H}(\tau,\cdot)=\begin{cases} F^{\alpha} & \text{ for $\tau=0$,} \\ F^{\beta} & \text{ for $\tau=1$,} \end{cases}
\end{equation*} 
be a homotopy between them, independent of $\tau$ near $0$ and $1$. 

Set 
\begin{gather*}
K: C^{\alpha} \to C^{\beta},\\ 
x+y\cdot T\mapsto (P+Q\cdot T)(x+y\cdot T)=Px +(Py +Qx)\cdot T. 
\end{gather*}
$K$ is a cochain map, if 
\begin{gather*}
K d^{\alpha}_{S^1} = d^{\beta}_{S^1} K\\
 \Longleftrightarrow \quad Pd=dP\quad \text{and} \quad Qd+PR_1^{\alpha}=dQ+R_1^{\beta}P. 
\end{gather*}
Therefore, we have to define $P:\C_k(f) \to \C_k(f)$ and $Q:\C_k(f)\to C_{k-2}(f)$, such that both equations are fulfilled.

Use a smooth function $h: [0,1]\to \mathbb{R}, \ \tau \mapsto h(\tau)$ to obtain a Morse-Bott function $H:[0,1]\times W\to \mathbb{R}$ (cf. Subsection $3.2.1$) with critical set
\begin{equation*}
\C(H)= \{0\}\times \C(F^{\alpha}) \cup \{1\}\times\C(F^{\beta}).
\end{equation*}
Recall the notation for the critical submanifolds in Subsection $3.2.1$ and let 
\begin{gather*}
A^{\alpha}_x:=\{0\}\times A_x \quad  B^{\alpha}_{y}:=\{0\}\times B_y, \\
A^{\beta}_x:=\{1\}\times A_x \quad  B^{\beta}_{y}:=\{1\}\times B_y
\end{gather*}
denote the critical submanifolds of $H$. Define a vector field $V^{\tau}$ on $[0,1]\times W$ by
\begin{equation*}
V^{\tau}:= \left(\frac{dH}{d\tau}\partial_{\tau},\nabla_g H,0,\frac{dH}{dr}\partial_r\right)
\end{equation*}
and observe that the flow lines of $V^{\tau}$ living in the two copies of $W$ at $\tau=0$ and $\tau=1$ are precisely those flow lines used to define the operators $d,\ R_1^{\alpha}$ and $R_1^{\beta}$, respectively.

The (un-)stable manifolds associated to $V^{\tau}$ have the following dimensions (cf. Lemma \ref{dimsf1}):
\begin{gather*}
\dim W^u(A^{\alpha}_{x})= n-\mu(x)+3,\quad \dim W^s(A^{\alpha}_{x})=\mu(x)+1,\\
\dim W^u(B^{\alpha}_{y})= n-\mu(y)+2,\quad \dim W^s(B^{\alpha}_{y})=\mu(y)+2.\\
\\
\dim W^u(A^{\beta}_{x})= n-\mu(x)+2,\quad \dim W^s(A^{\beta}_{x})=\mu(x)+2,\\
\dim W^u(B^{\beta}_{y})= n-\mu(y)+1,\quad \dim W^s(B^{\beta}_{y})=\mu(y)+3.
\end{gather*}
Define the connecting spaces of flow lines from $A^{\alpha}_{x}$ to $A^{\beta}_{y}$ and from $B^{\alpha}_{x}$ to $B^{\beta}_{y}$ by
\begin{gather*}
\M^{\alpha \beta}_A(x,y):= W^u(A^{\alpha}_{x})\cap W^s(A^{\beta}_{y}), \\ 
\M^{\alpha \beta}_B(x,y):= W^u(B^{\alpha}_{x})\cap W^s(B^{\beta}_{y}),
\end{gather*} 
and observe that 
\begin{gather*}
\tilde{H}(0,p,s,0)=\tilde{H}(1,p,s,0)=f(p), \\
\tilde{H}(0,p,s,1)=\tilde{H}(1,p,s,1)=f(s.p).
\end{gather*} 
Therefore, there are free $\mathbb{R}$- and $S^1$-actions on both spaces and an isomorphism between them (cf. Proposition \ref{ms1}). Repeating the steps in Section $4.1$ we conclude that the associated moduli space $\widetilde{\M}^{\alpha \beta}_P(x,y)$ is a smooth manifold of dimension $\mu(y)-\mu(x)$.

Now define $P$ on a generator $x\in \C(f)$ by 
\begin{equation*}
x\mapsto \sum_{\mu(y)=\mu(x)}|\widetilde{\M}^{\alpha \beta}_P(x,y)|\bmod{2}\cdot y.
\end{equation*}
If $\mu(y)=\mu(x)+1$, then $\widetilde{\M}^{\alpha \beta}_P(x,y)$ is a one-dimensional manifold. From the compactness and gluing arguments of Chapter $4$ it follows, that it is compactifiable using broken flow lines $(u,v)$ of order two (higher order broken flow lines are not possible - cf. the last part in the proof of the Compactness Theorem). Here $u\in \M^{\alpha}_0(x,z)$ and $v\in \M^{\alpha \beta}_P(z,y)$ with $\mu(z)=\mu(x)+1$, or $u\in \M^{\alpha \beta}_P(x,z)$ and $v\in \M^{\beta}_0(z,y)$ with $\mu(z)=\mu(x)$. Since the subsets $[0,1]\times M \times S^1\times \{0\}$ and $[0,1]\times M \times S^1\times \{1\}$ of $[0,1]\times W$ are both flow-invariant, there are no other flow lines to which a sequence in $\widetilde{\M}^{\alpha \beta}_P(x,y)$ might converge. As in the proof of $d_{S^1}^2=0$ this implies 
\begin{equation*}
P d+dP=0 \Longleftrightarrow P d=d P.
\end{equation*}

On the other hand, the moduli space $\widetilde{\M}_Q^{\alpha \beta}(x,y):=W^u(A^{\alpha}_{x})\ti W^s(B^{\beta}_{y})/\mathbb{R}$ of flow lines from $A^{\alpha}_{x}$ to $B^{\beta}_{y}$ is a manifold of dimension $\mu(y)-\mu(x) +2 $. Therefore, we define $Q$ by 
\begin{equation*}
x \mapsto \sum_{\mu(y)=\mu(x)-2}|\widetilde{\M}_Q^{\alpha \beta}(x,y)|\bmod{2}\cdot y.
\end{equation*}

In the same way as above we conclude, that for $\mu(y)=\mu(x)-1$ the compactification of $\widetilde{\M}_Q^{\alpha \beta}(x,y)$ is given by adding broken flow lines $(u,v)$ of order two. Here we encounter four possible cases, depending on at which critical submanifold of $[0,1]\times W$ ``breaking up'' occurs, i.e. at $(\tau,r)=(0,0),(1,0),(0,1)$ or $(1,1)$: 
\begin{align*}
(u,v)&\in \M^{\alpha}_1(x,z) \times \M^{\alpha \beta}_P(z,y), \ \mu(z)=\mu(x)-1,\ (\tau,r)=(0,1), \\
(u,v)&\in \M^{\alpha \beta}_P(x,z)\times \M_1^{\beta}(z,y), \ \mu(z)=\mu(x),\ (\tau,r)=(1,0),\\
(u,v)&\in \M^{\alpha}_0(x,z)\times \M^{\alpha \beta}_Q(z,y), \ \mu(z)=\mu(x)+1,\ (\tau,r)=(0,0), \\
(u,v)&\in \M^{\alpha \beta}_Q(x,z)\times \M^{\beta}_0(z,y), \ \mu(z)=\mu(x)-2,\ (\tau,r)=(1,1).
\end{align*}

From gluing arguments it follows, that these are the only possible cases. Thus, 
\begin{equation*}
P R^{\alpha}_1+R_1^{\beta} P+Qd+dQ=0.
\end{equation*}

This shows that $K$ is a cochain map.
\\

Note that if $\tilde{H}$ is the constant homotopy, then obviously $P$ is the identity map. Moreover, in this case $\widetilde{\M}_Q^{\alpha \beta}(x,y)$ consists of flow lines $(a,b):\mathbb{R}\to [0,1]\times W$ with 
\begin{align*}
\dot{a}&=\frac{dh}{d\tau}\circ a,\quad a^-=0,\ a^+=1, \\ 
\dot{b}&=V\circ b,\quad b^-\in A_x,\ b^+\in B_y.
\end{align*} 
Corollary \ref{dims} implies that $\widetilde{\M}_Q^{\alpha \beta}(x,y)=\emptyset$, i.e. $Q$ is the zero map and therefore $K=P+Q\cdot T$ is the identity map on cochains.

\begin{figure}[hbt]
	\centering
\includegraphics[width=1.00\textwidth]{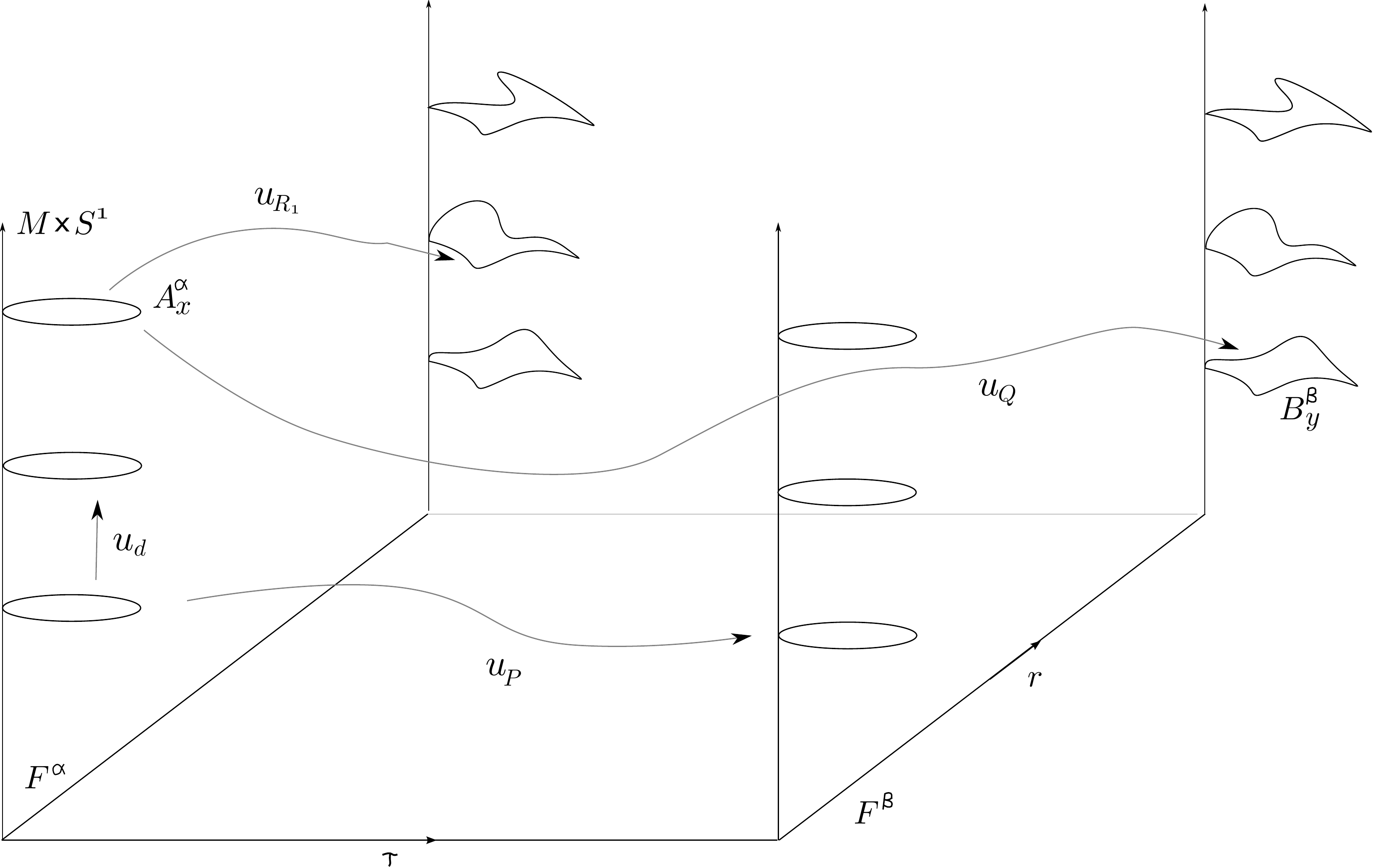}
	\caption{Flow lines $u$ on $[0,1]\times W$}
	\label{bild7}
\end{figure}

Now let $F^{\alpha},F^{\beta},F^{\gamma},F^{\delta}$ be four $S^1$-families of homotopies between $f$ and $f\circ \sigma_s$ and let 
\begin{equation*}
\tilde{\mathcal{H}}: [0,1]\times [0,1]\times W \to \mathbb{R},\ \tilde{\mathcal{H}}(\nu,\tau,\cdot)=\begin{cases} F^{\alpha} & \text{ for $(\nu,\tau)=(0,0)$,} \\ F^{\beta} & \text{ for $(\nu,\tau)=(1,0)$,} \\ F^{\gamma} & \text{ for $(\nu,\tau)=(0,1)$,} \\ F^{\delta} & \text{ for $(\nu,\tau)=(1,1)$,}  \end{cases}
\end{equation*} 
be a homotopy of homotopies, independent of $\nu$ and $\tau$ near $(0,0),(1,0),(0,1)$ and $(1,1)$. Let $K^{ij}$ denote the cochain map between $C^i$ and $C^j$ ($i,j \in \{\alpha,\beta,\gamma,\delta\}$), induced by the homotopy between $F^i$ and $F^j$ given by the corresponding restriction of $\tilde{\mathcal{H}}$ to a face of $[0,1]^2$. We construct a cochain homotopy $\mathcal{K}$ between $K^{\beta \delta} K^{\alpha \beta}$ and $K^{\gamma \delta} K^{\alpha \gamma}$, i.e a cochain map of degree $-1$:
\begin{gather*}
\mathcal{K}: C^{\alpha} \to C^{\delta}, \\
\mathcal{K}:=\mathcal{P}+\mathcal{Q}\cdot T,
\end{gather*}
satisfying
\begin{align*}
\mathcal{K}  d^{\alpha}_{S^1} + d^{\delta}_{S^1} \mathcal{K}  = K^{\beta \delta} K^{\alpha \beta} - K^{\gamma \delta} K^{\alpha \gamma}, 
\end{align*}
which is equivalent to
\begin{gather*}
\mathcal{P}d+d\mathcal{P}=P^{\beta \delta}P^{\alpha \beta}- P^{\gamma \delta}P^{\alpha \gamma} \\
\text{and} \\
\mathcal{P}R_1^{\alpha}+R_1^{\delta}\mathcal{P} + \mathcal{Q}d +d\mathcal{Q}=P^{\beta \delta}Q^{\alpha \beta}+ Q^{\beta \delta}P^{\alpha \beta}- P^{\gamma \delta}Q^{\alpha \gamma}- Q^{\gamma \delta}P^{\alpha \gamma}.
\end{gather*}

Repeating the steps in the construction of the map $K$, we obtain a Morse-Bott function $\mathcal{H}:[0,1]^2\times W\to \mathbb{R}$ with critical set 
\begin{equation*}
\{(0,0)\}\times \C(F^{\alpha}) \cup \{(1,0)\}\times \C(F^{\beta}) \cup \{0,1\}\times \C(F^{\gamma}) \cup \{(1,1)\}\times \C(F^{\delta}).
\end{equation*}
The stable and unstable manifolds of the vector field 
\begin{equation*}
V^{\nu \tau}:=\left(\frac{d\mathcal{H}}{d\nu}\partial_{\nu},\frac{d\mathcal{H}}{d\tau}\partial_{\tau},\nabla_g \mathcal{H},0,\frac{d\mathcal{H}}{dr}\partial_r\right)
\end{equation*}
have the following dimensions ($i=\{\beta,\gamma\}$):
\begin{gather*}
\dim W^u(A^{\alpha}_x)=n-\mu(x)+ 4,\quad \dim W^s(A^{\alpha}_x)= \mu(x)+1,\\
\dim W^u(B^{\alpha}_y)=n-\mu(y)+ 3,\quad \dim W^s(B^{\alpha}_y)= \mu(y)+2.\\
\\
\dim W^u(A^i_x)=n-\mu(x)+ 3,\quad \dim W^s(A^i_x)= \mu(x)+2,\\
\dim W^u(B^i_y)=n-\mu(y)+ 2,\quad \dim W^s(B^i_y)= \mu(y)+3.\\
\\
\dim W^u(A^{\delta}_x)= n-\mu(x)+2,\quad \dim W^u(A^{\delta}_x)= \mu(x)+3,\\
\dim W^u(B^{\delta}_y)= n-\mu(y)+1,\quad \dim W^u(B^{\delta}_y)= \mu(y)+4.
\end{gather*}

Let $\M^{\alpha \delta}_{\mathcal{P}}(x,y)$ denote the space of flow lines from $A^{\alpha}_x$ to $A^{\delta}_y$ (or equivalently with $A$ replaced by $B$). As above we conclude that this space is endowed with free $\mathbb{R}$- and $S^1$-actions and the quotient $\widetilde{\M}^{\alpha \delta}_{\mathcal{P}}(x,y)$ is a manifold of dimension $\mu(y)-\mu(x)+1$. Therefore, we define $\mathcal{P}:\C_k(f) \to \C_{k-1}(f)$ by
\begin{equation*}
x \mapsto \sum_{\mu(y)=\mu(x)-1} |\widetilde{\M}^{\alpha \delta}_{\mathcal{P}}(x,y)|\bmod{2}\cdot y.
\end{equation*}

For $\mu(y)=\mu(x)$ the moduli space $\widetilde{\M}^{\alpha \delta}_{\mathcal{P}}(x,y)$ is one-dimensional. Its compactification is given by adding broken flow lines $(u,v)$ of the following types:
\begin{align*}
(u,v)&\in \M_0^{\alpha}(x,z)\times \M^{\alpha \delta}_{\mathcal{P}}(z,y), \ \mu(z)=\mu(x)+1, \\
(u,v)&\in \M^{\alpha \delta}_{\mathcal{P}}(x,z)\times \M_0^{\delta}(z,y), \ \mu(z)=\mu(x)-1, \\
(u,v)&\in \M^{\alpha \beta}_P(x,z)\times \M^{\beta \delta}_P(z,y), \ \mu(z)=\mu(x), \\
(u,v)&\in \M^{\alpha \gamma}_P(x,z)\times \M^{\gamma \delta}_P(z,y), \ \mu(z)=\mu(x). 
\end{align*}
This shows 
\begin{equation*}
\mathcal{P}d+d\mathcal{P}=P^{\beta \delta} P^{\alpha \beta} - P^{\gamma \delta} P^{\alpha \gamma}.
\end{equation*}

Regarding flow lines from $A^{\alpha}_x$ to $B^{\delta}_y$, the corresponding moduli space $\widetilde{\M}^{\alpha \delta}_{\mathcal{Q}}(x,y)$ has dimension $\mu(y)-\mu(x)+3$. We define $\mathcal{Q}:\C_k(f) \to \C_{k-3}(f)$ by
\begin{equation*}
x \mapsto \sum_{\mu(y)=\mu(x)-3} |\widetilde{\M}^{\alpha \delta}_{\mathcal{Q}}(x,y)|\bmod{2}\cdot y.
\end{equation*}
If $\mu(y)=\mu(x)-2$, then $\widetilde{\M}^{\alpha \delta}_{\mathcal{Q}}(x,y)$ is a one-dimensional manifold, compactifiable with broken flow lines $(u,v)$ of order two. There are the following types of broken flow lines, depending on which of the vertices of $[0,1]^2$ they pass:
\begin{align*}
(u,v)&\in \M_1^{\alpha}(x,z)\times \M^{\alpha \delta}_{\mathcal{P}}(z,y), \ \mu(z)=\mu(x)-1, \\
(u,v)&\in \M^{\alpha \delta}_{\mathcal{P}}(x,z)\times \M_1^{\delta}(z,y), \ \mu(z)=\mu(x)-1, \\
(u,v)&\in \M^{\alpha \delta}_{\mathcal{Q}}(x,z)\times \M_0^{\delta}(z,y), \ \mu(z)=\mu(x)-3, \\
(u,v)&\in \M^{\alpha}_0(x,z)\times \M^{\alpha \delta}_{\mathcal{Q}}(z,y), \ \mu(z)=\mu(x)+1, \\
(u,v)&\in \M_Q^{\alpha \beta}(x,z)\times \M_P^{\beta \delta}(z,y), \ \mu(z)=\mu(x)-2, \\
(u,v)&\in \M^{\alpha \gamma}_Q(x,z)\times \M^{\gamma \delta}_P(z,y), \ \mu(z)=\mu(x)-2, \\
(u,v)&\in \M^{\alpha \beta}_P(x,z)\times \M_Q^{\beta \delta}(z,y), \ \mu(z)=\mu(x), \\
(u,v)&\in \M^{\alpha \gamma}_P(x,z)\times \M^{\gamma \delta}_Q(z,y), \ \mu(z)=\mu(x). 
\end{align*}

Hence, 
\begin{equation*}
\mathcal{P}R_1^{\alpha}+R_1^{\delta}\mathcal{P} + \mathcal{Q}d +d\mathcal{Q}=P^{\beta \delta}Q^{\alpha \beta}+ Q^{\beta \delta}P^{\alpha \beta}- P^{\gamma \delta}Q^{\alpha \gamma}- Q^{\gamma \delta}P^{\alpha \gamma},
\end{equation*}
and therefore $\mathcal{K}$ is a cochain homotopy. 
\\

In the special case $F^{\alpha}=F^{\gamma}$, $F^{\beta}=F^{\delta}$, it follows that for two homotopies $\tilde{\mathcal{H}}\restriction_{[0,1]\times \{0\}\times W}$ and $\tilde{\mathcal{H}}\restriction_{[0,1]\times \{1\} \times W}$ from $F^{\alpha}$ to $F^{\beta}$ the induced cochain maps $K^{\alpha \beta}_0$ and $K^{\alpha \beta}_1$ between $C^{\alpha}$ and $C^{\beta}$ are cochain homotopic:
\begin{align*}
\mathcal{K}  d^{\alpha}_{S^1} + d^{\beta}_{S^1} \mathcal{K} & = K^{\beta \beta} K^{\alpha \beta}_0 - K^{\alpha \beta}_1 K^{\alpha \alpha} \\
&= K^{\alpha \beta}_0 - K^{\alpha \beta}_1.
\end{align*}

Finally, setting $\gamma=\alpha$ and $\delta=\alpha$, we conclude that $P^{\beta \alpha}P^{\alpha \beta}$ is cochain homotopic to the identity:
\begin{align*}
\mathcal{K}  d^{\alpha}_{S^1} + d^{\beta}_{S^1} \mathcal{K} & = K^{\beta \alpha} K^{\alpha \beta} - K^{\alpha \alpha} K^{\alpha \alpha} \\
&= K^{ \beta \alpha}K^{\alpha \beta}  - 1.
\end{align*}
\\

Therefore, all three properties of the continuation principle mentioned in Chapter $2$ are satisfied. Thus, $H_*(C^{\alpha})\cong H_*(C^{\beta})$. This finishes the proof.
\end{proof}

\section{Summary}
\label{sum} 

So far we have shown that the $S^1$-equivariant Morse cohomology equals the (ordinary) equivariant cohomology for smooth closed $S^1$-manifolds $M$ of dimension less than $3$. It remains to show the case $m>1$ in the last theorem of Chapter $5$. For this one hast to dig a little deeper into the structure of the bundles $\pi_m$ to generalize and prove the parallel transport property mentioned in the proof of Theorem $5.4$. Then one needs to show the equivalence of the ``smooth'' and ``non-smooth'' pictures of $k$-jump flow lines and independence of the chosen homotopy. But this seems to be just a rather technical issue, as well as the question of extending the whole construction to a broader class of $S^1$-spaces (i.e. dropping the assumption that $M$ is closed).

More interesting is the question of orientability, i.e. $S^1$-equivariant Morse cohomology with $\mathbb{Z}$-coefficients. As mentioned in Remark \ref{orient} the moduli spaces of $k$-jump flow lines lack a natural concept of orientation - if there is no fruitful geometrical idea, maybe one needs to resolve this by using a Floer-type functional analytic approach (cf. Salamon \cite{Sal} or Schwarz \cite{Sch}).

Another question is the following: Is there a similiar way to define $HM^*_G(M)$ for other Lie groups $G$. This is not clear at all, because $S^1$ is the only connected Lie group with its subgroups being either $S^1$ itself or discrete. As a consequence the orbits of a $G$-action will in general have very different structures (e.g. dimensions). Moreover, our construction relies heavily on the struture of $BS^1=\mathbb{CP}^{\infty}$. $G=T^n=(S^1)^n$ is a special case which deserves some attention, but there is definitely a lot of work to do to generalize the ideas presented in this thesis.

As mentioned in the introduction, a related interesting object is $S^1$-equivariant Floer cohomology and, in the same way classical Morse theory served as a toy model for Floer theory, one might expect the same for equivariant Morse theory. For this thesis we have also studied the Floer-type approach to Morse homology: Transversality and compactness are not hard to show using Fredholm theory on Banach bundles, but it gets tricky with the gluing map and the orientation concept - this is a interesting topic for further studies in this direction.








\end{document}